\pdfoutput=1
\documentclass[11pt]{article}
\usepackage{amsmath}
\usepackage{amssymb}
\usepackage{amsthm}
\usepackage{pdfpages} 
\usepackage[pdftex,colorlinks]{hyperref}
\newcommand {\R}{\mathbb R}

\newcommand{\e}{\mathbb E}

\theoremstyle{definition}\newtheorem{thm}{Theorem}
\theoremstyle{definition}\newtheorem{lem}[thm]{Lemma}
\theoremstyle{definition}\newtheorem{cor}[thm]{Corollary}
\theoremstyle{definition}
\theoremstyle{definition}\newtheorem{defi}[thm]{Definition}
\theoremstyle{definition}\newtheorem{rem}[thm]{Remark}
\theoremstyle{definition}\newtheorem{prop}[thm]{Proposition}
\theoremstyle{definition}
\theoremstyle{definition}
\theoremstyle{definition}\newtheorem{assum}{Assumption}
\theoremstyle{definition}
\theoremstyle{definition}\newtheorem{ex}{Example}
\numberwithin{equation}{section}

\newcommand{\PP}{\mathbb P}

\newcommand{\B}{{\cal B}}

\newcommand{\Z}{\mathbb Z}

\newcommand{\ind}{\mathbf 1}

\newcommand{\bbN}{\Bbb{N}}

\newcommand{\re}{\texttt{restart}~}
\newcommand{\ch}{\texttt{checkpointing}~}
\newcommand{\res}{\texttt{restart}}
\newcommand{\che}{\texttt{checkpointing}}

\newcommand{\E}{\e}

\newcommand{\iid}{\hbox{i.i.d.}}
\newcommand{\bL}{\textbf{L}}

\begin{document}

\author{Antonio Sodre\thanks{asodre@math.utexas.edu} \\{\small University of Texas at Austin}}
\title{Asymptotic efficiency of restart and checkpointing}
\date{}
\maketitle
\begin{abstract}
Many tasks are subject to failure before completion. Two of the most common failure recovery strategies are \re and \break \che. Under \res, once a failure occurs, it is restarted from the beginning. Under \che, the task is resumed from the preceding checkpoint after the failure. We study asymptotic efficiency of \re for an infinite sequence of tasks, whose sizes form a stationary sequence. We define asymptotic efficiency as the limit of the ratio of the total time to completion in the absence of failures over the total time to completion when failures take place. Whether the asymptotic efficiency is positive or not depends on the comparison of the tail of the distributions of the task size and the random variables governing failures. Our framework allows for variations in the failure rates and dependencies between task sizes. We also study a similar notion of asymptotic efficiency for \ch when the task is infinite $\hbox{a.s.}$ and the inter-checkpoint times are \iid. Moreover, in \che, when the failures are exponentially distributed, we prove the existence of an infinite sequence of universal checkpoints, which are always used whenever the system starts from any checkpoint that precedes them. 
 \end{abstract}


\section*{Introduction}

In many situations, such as the execution of a computer program, the copy of a file from a remote location using a protocol such as FTP or HTTP, channel reservation in cognitive radio networks and others, tasks are subject to failures. \texttt{Restart} and \ch are two of the most common ways to take into account failures in these context (see, among others, \cite{krishnamurthy2001web},\cite{nair2010file}, and \cite{castellanos2012channel}).

In \res, whenever a failure occurs, as the name suggests, the task is restarted. Accordingly, the \textit{actual} time to completion is possibly larger than the \textit{ideal} time. The latter is defined as the time for completion without failures. In \che, the task is partitioned: when a failure occurs, it is resumed from the last element of the partition before the failure. 

Here is a basic description of \res. Let $D$ be the ideal task time. If no failure occurs, the actual time to complete the task is just $D$. If a failure occurs at $L_0<D$, the task is restarted. Suppose there are $\nu>0$ failures before the task is completed. Then the actual time is given by $T^{R}=\sum_{i=0}^{\nu} L_i+D$. Failures are modeled by a sequence of \iid~random variables $\{L_n\}_{n\geq 0}$, named \emph{failure} times. The one-task \re model is studied in \cite{restart1}, \cite{restart2}, and \cite{asmussen2014failure} for a random variable $D$ with unbounded support (see Figure 1). Section \ref{localtime} introduces the formalism for the one-task \re model. 

In the one-task case, the actual time, $T^R$, is heavy-tailed, even when the ideal time and the failure time have light tails. Moreover, the actual task time may have infinite expectations, even if both $D$ and $L_0$ do not, depending on the comparison of the tail distributions of $D$ and $L_0$ \cite{asmussen2014failure}. 

We extend the literature on \re by considering an infinite sequence of tasks, $\{D_n\}_{n\geq 0}$, introducing the concept of asymptotic efficiency. Let $T^R_n$ be the actual time of task $n$. We define asymptotic efficiency as
\begin{align}
\label{efficiency1}
e=\lim_{N\to \infty}\frac{\sum_{n=0}^{N-1} D_n}{\sum_{n=0}^{N-1} T^R_n}, 
\end{align}
whenever the limit exists \hbox{a.s.}. The system is inefficient when $e=0$.

In this sequential \re model, the ideal times is given by the distance between points of a simple stationary point process in $\mathbb{R}$. Such a point process can be seen as a random discrete sequence of distinct elements on $\R$, $\{X_n\}_{n\in \Z}$, such that $X_n<X_{n+1}$ for all $n$. The sequence of tasks sizes is given by $D_0=X_1-X_0$, $D_1=X_2-X_1$ and so on. We mark the point $X_n$ with an \iid~sequence, $\{L_{n,i}\}_{i\geq 1}$, capturing the failure times of the $n^{th}$-task. We present the point process setting for modeling task sizes and failures in Section \ref{generalrestartandcheckpoint}. 

We prove that asymptotic efficiency exists when the point process is stationary, the failure sequence $\{L_{n,i}\}_{i\geq 1}$ is independent of $D_n$, and under some integrability conditions. We do not require the sequence $\{D_n\}_{n\geq 0}$ to be \iid. In fact, our set-up allows for variations in the failure rates and dependencies between task sizes.

Moreover, we give conditions under which the asymptotic efficiency is positive or zero. Two special cases are considered: Markov renewal process, and the case in which there is a chance that tasks need to be repeated after completion (Section \ref{restartdiscussion}). 

The \ch model can be described as follows \cite{asmussen2014failure}. Again, we have a random a task $D$ of random length with infinite support, but finite a.s.. We partition $[0,D]$ into $k$ intervals and label the endpoint of the $l^{th}$ interval by $X^l$. We call $\{X^l\}_{l=1}^{k-1}$ the set of checkpoints. Once a checkpoint is reached and a failure occurs, the task is resumed from the latest checkpoint before the failure. More precisely, if a failure occurs before the first checkpoint, i.e., $L_1<X^1$, the task is resumed from the beginning. If $L_1>D$, i.e., there are no failures, the actual time to completion is simply $D$. Otherwise, if $X^1<L_1<D$, we check the partition in which $L_1$ falls. If $X^l\leq L^1<X^{l+1}$, the task is resumed from $X^l$ and the time spent so far is $L_1$. In that case, we start the clock again, representing it by the random variable $L_2$. If $L_2<X^{l+1}-X^l$, the task does not leave the checkpoint $X^l$. Otherwise, we verify which checkpoint was reached or whether the task was completed. We repeat this procedure until the task is completed. Assume that there are $\tau>0$ failures until completion. Then, the actual time is given by $T^{C}=\sum_{i=1}^\tau L_i+(D-X_\alpha)$, where $\alpha\in \{1,\ldots,k\}$ is the last checkpoint visited (see Figure 2).

Regarding the sequential \ch considered here, we define and study a notion of asymptotic efficiency, in a similar way to (\ref{efficiency1}). We consider a unique task, which is $\hbox{a.s.}$ infinite, and the distances between checkpoints are given by the inter-arrivals of a point process. We give the precise definition of asymptotic efficiency for \ch in Section \ref{generalrestartandcheckpoint}. We give a general condition for the existence of the asymptotic efficiency when the point process is a marked renewal process. 
\begin{figure}
\centering
\includegraphics[width=0.5\textwidth]{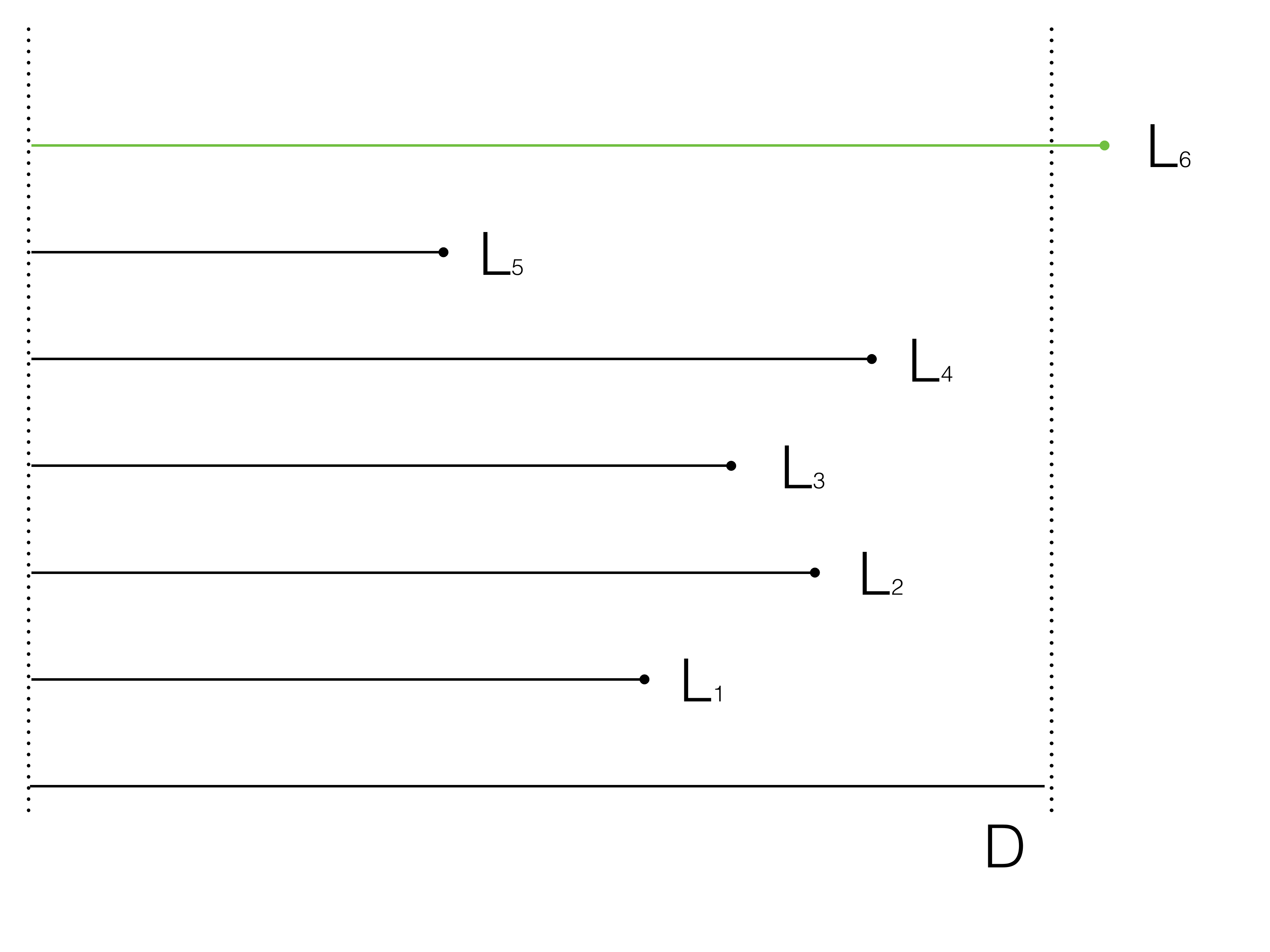}
\caption{An instance of \res. The task size is $D$. Five attempts take place before the task is completed, i.e., $\nu=5$. The time spent on each attempt is given by $L_1,\ldots,L_5$. In the sixth attempt the task is completed. The \emph{actual} time spent on completing the task is then $T^R=\sum_{i=1}^5L_i+D$.}
\end{figure}
\begin{figure}
\centering
\includegraphics[width=0.5\textwidth]{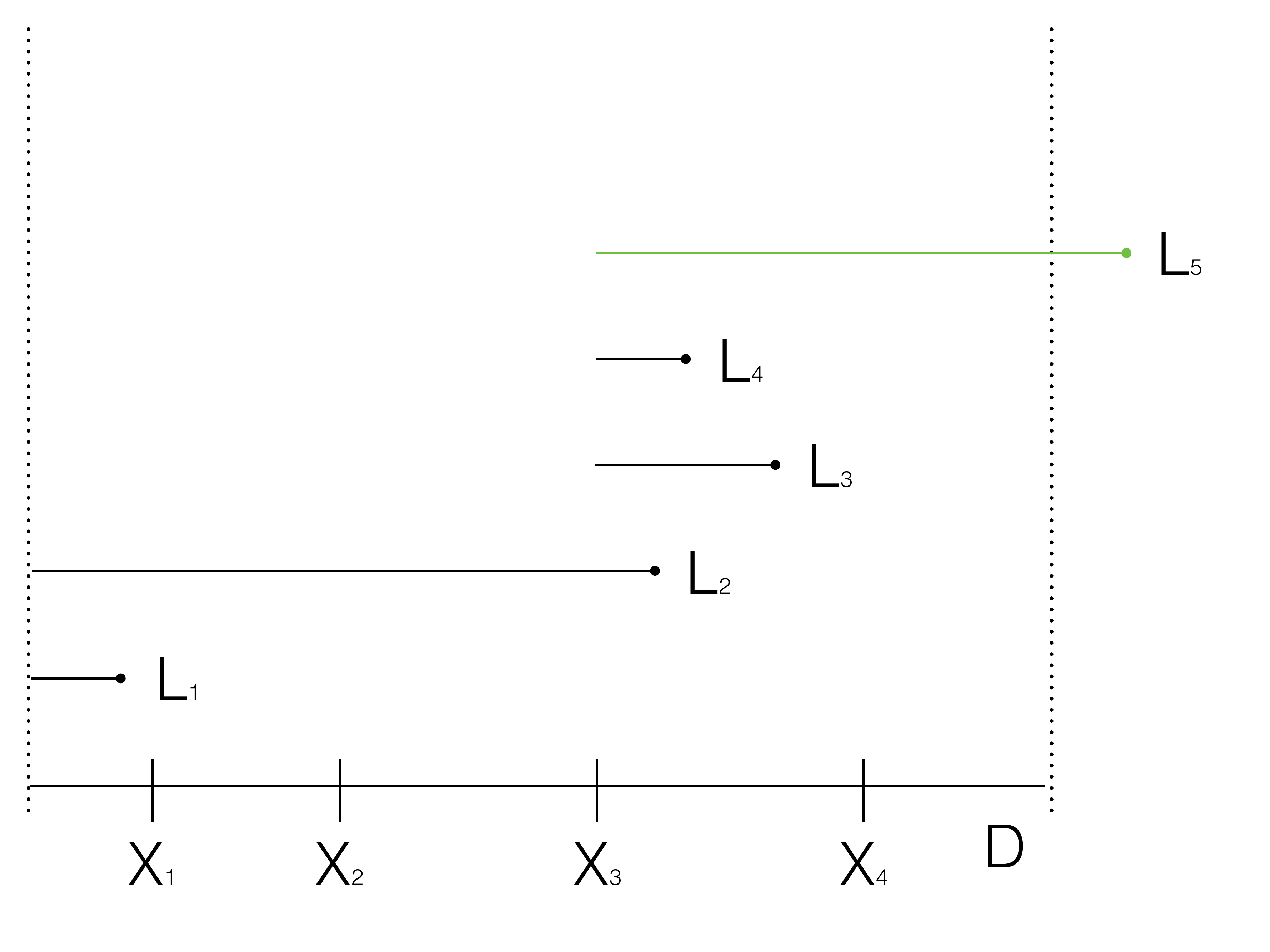}
\caption{An instance of \che. There is one failure before the first checkpoint, the second failure takes place after the third checkpoint and two more failures happen before the last checkpoint is surpassed. The \emph{actual} time till completion is given by $T^{C}=\sum_{i=1}^4L_i+(D-X_3)$.}
\end{figure}

Moreover, in the renewal process model with exponentially distributed failure times, we show the existence of an infinite subsequence of \emph{universal checkpoints}. If we start the system at any checkpoint preceding a universal checkpoint, the system will activate the latter a.s..

Section \ref{localtime} reviews the \emph{actual} time for one-task \re and \break \che, and gives the conditions under which the \emph{actual} time has finite moments. Section \ref{generalrestartandcheckpoint} presents a unified framework to study the asymptotic efficiency for both sequential \re and \che. Section \ref{restartmainresults} presents our main results for sequential \res. Section \ref{checkpointsection} does the same for sequential \che. Section \ref{restartdiscussion} discusses some extensions. The appendix contains a technical proof.


\section{One task \re and \ch}
\label{localtime}

In this section we recall known results on how to compute the actual time for \re and \ch with a finite a.s. task. This type of model was extensively studied in  \cite{restart1}, \cite{restart2}, and \cite{asmussen2014failure}. Here, our main result is whether the actual \re and \ch times have an infinite expectation depends on the tail comparison of $D$ and $L_0$, the former being the variable capturing failures. We build on this result in our study of the asymptotic efficiency in the next section. 

Let $D$ be a random variable in $\R^+$ which represents the ideal task time in both the \re model or the time up to the first checkpoint under the \ch model. Consider a sequence of \hbox{i.i.d.} random variables $\{L_n\}_{n\geq 0}$ defined on the same probability space $(\Omega,\mathcal{F},\PP)$ as $D$. 
Define 
\begin{equation*}
\tau=\inf\{k\geq 0: L_k>D\}.
\end{equation*}

The actual time taken to complete the task $D$ under \re is given by
\begin{equation*}
\label{time1}
T^R=\sum_{i=0}^{\tau-1} L_i+D,
\end{equation*}
and the actual time the system takes to pass the first checkpoint is 
\begin{equation*}
\label{time2}
T^C=\sum_{i=0}^{\tau} L_i.
\end{equation*}

\begin{assum}
\label{integrabilityassumption} 
$D$ and $L_0$ are integrable and independent random variables with \emph{right-unbounded support}, i.e, $\PP[D>x],\PP[L_0>x]>0$ for all $x\in [0,\infty)$. 
\end{assum}

\begin{defi}
\label{lighterthan}
Let $V$ and $W$ be random variables with right-unbounded support, defined on the same probability space $(\Omega,\mathcal{F},\PP)$. We say then that $V$ has a $\PP-$tail heavier than $W$ if there exists $z_0$ such that $\PP[V>z]\geq \PP[W>z]$ for all $z\geq z_0$. In the same vein, we say that $V$ has a strict $\PP-$tail heavier than $W$ if there exists $z_0$ and $\epsilon>0$ such that $\PP[V>z]\geq \PP[W>z]^{\epsilon}$ for all $z\geq z_0$.
\end{defi}

\begin{thm}
\label{assymptotics1} 
 Under Assumption \ref{integrabilityassumption}, 
\begin{align*}
\E[T^R]&=\E[D]+\int_0^{\infty}\frac{\E[L_0\textbf{1}\{L_0\le z\}]}{\PP[L_0>z]}f_D(dz)
\end{align*}
and
\begin{align*}
\E[T^C]&=\E[L_0]\int_0^{\infty} \frac{1}{\PP[L_0>z]}f_D(dz),
\end{align*}
where $f_D$ is the distribution of $D$. 
Moreover, $\E[T^R],\E[T^C]=\infty$ if $D$ has a $\PP-$tail heavier than $L_0$ and $\E[T^R],\E[T^C]<\infty$ if $L_0$ has a strict $\PP-$tail heavier than $D$. 
\end{thm} 
The proof can be found in Appendix \ref{proofsappendix}. 

As an application of Theorem \ref{assymptotics1}, suppose $L_0\sim \exp(\lambda_l)$ and $D\sim \exp(\lambda_d)$. Then $\E[T^R],\E[T^C]=\infty$ if and only if $\lambda_l\geq \lambda_d$. Notice also that when $D\overset{(d)}{=}L_0$, $\E[T^R],\E[T^C]=\infty$. 

A random variable $Z$ is said to be heavy-tailed if for all $\gamma>0$,
\begin{align*}
\lim_{t\to \infty} e^{\gamma t}\PP[Z>t]=\infty,
\end{align*}
and light-tailed if there exists $\gamma>0$ such that the above limit is finite. By direct manipulations, one gets the following corollary of Theorem \ref{assymptotics1}. 

\begin{cor}
\label{renewaliidtails} 
Under Assumption \ref{integrabilityassumption}:
\begin{enumerate}
\item If $D$ is heavy-tailed and $L_0$ is light-tailed, $\E[T^C],\E[T^R]=\infty$;
\item If $L_0$ is heavy-tailed and $D$ is light-tailed, $\E[T^C],\E[T^R]<\infty$. 
\end{enumerate} 
\end{cor}


\section{Sequential \re and \che}
\label{generalrestartandcheckpoint} 
The goal of this section is to define the asymptotic efficiency under \break \re (resp. \che) when there is a sequence of tasks (resp. a sequence of checkpoints) whose ideal times to completion (resp. distance between checkpoints) are given by the inter-arrival times of a stationary point process. We call the models introduced in this section {\em sequential} \re and \che.  
\subsection{Point process and stationarity}
First, we briefly review the necessary concepts in point process theory. For a more complete treatment on the subject see \cite{daley1},\cite{daley2},\cite{mecke} among others. Consider a general probability space endowed with a measurable flow \break $(\Omega,\mathcal{F},\PP,\{\theta_t\}_{t\in \R})$. Let $\textbf{N}(\R\ltimes (\R^+)^{\bbN})$ be the set of counting measures on $\R$ with marks in $(\R^+)^{\bbN}$. An element of 
$\textbf{N}(\R\ltimes (\R^+)^{\bbN})$ is of the form $\psi=\sum_{n\in \Z}\delta_{(X_n,K_n)}(\cdot)$, in which $\delta_{Z}(\cdot)$ is the Dirac measure with mass at $Z$, $X_n\in \R$, $K_n\in (\R^+)^{\bbN}$, and the sequence $\{X_n\}_{n\in \Z}$ does not have accumulation points. We say $K_n$ is the mark of $X_n$. For any $C\in \B(\R^+\times (\R^+)^{\bbN})$:
 $\psi(C)=\sum_{n\in \Z}\delta_{(X_n,K_n)}(C)$. We write $X_n\in \psi$ whenever $\psi(\{X_n,K_n\})\geq 1$.  

We equip $\textbf{N}(\R\ltimes (\R^+)^{\bbN})$ with the smallest $\sigma-$algebra $\mathcal{N}(\textbf{N}(\R\ltimes (\R^+)^{\bbN}))$ that makes the family of mappings 
\begin{align*}
\{\psi\mapsto \psi(C):C\in \mathcal{B}(\R^+\times (\R^+)^{\bbN}),~C~\hbox{bounded}\}
\end{align*}
measurable. A point process on $\R^+$ with marks in $(\R^+)^{\bbN}$ is a measurable mapping $\Phi:\Omega\to \textbf{N}(\R^+\ltimes (\R^+)^{\bbN})$.  

The realization of a marked point process in $\textbf{N}(\R\ltimes (\R^+)^{\bbN})$ corresponds to a sequence  $$\{X_n(\omega),K_n(\omega)\}_{n\in \Z}\subset \R\times (\R^+)^{\bbN}$$ such that $\{X_n(\omega)\}_{n\in \Z}$ has no accumulation points $\PP-$a.s.. We often write $X_n$ instead of $X_n(\omega)$. For all $C\in \mathcal{B}(\R\times \R^+)$, we let $\Phi(\omega,C)=\#\{(X_n,K_n)(\omega)\in C\}$. 
Moreover, we always label the points of $\Phi$ in $\R$ as follows:
\begin{align*}
\ldots\leq X_{-2}\leq X_{-1} \leq X_0 \leq 0 \leq X_1 \leq \ldots.	
\end{align*}

We assume $\Phi$ is $\theta_t-$compatible, i.e., for all $t\in \R$,
\begin{enumerate}
\item $\PP\circ (\theta_t)^{-1}=\PP$,
\item For all $C\in \B(\R)$ and $D\in \B((\R^+)^{\bbN})$:
\begin{align*}
\Phi(\theta_t\omega, C\times D)&=\Phi(\omega, (C+t)\times D).
\end{align*}
\end{enumerate} 
These, together with
\begin{align*}
\Phi(\omega, C\times (\R^+)^{\bbN})<\infty~\hbox{for all $C\in \B(\R)$ bounded, $\PP-\hbox{a.s.}$},
\end{align*}
makes $\Phi$ a stationary marked point process. 

\begin{rem}
\label{canonical}
It is most convenient for our purposes to take $\Omega$ to be $\textbf{N}(\R^+\ltimes (\R^+)^{\bbN})$ and $\mathcal{F}$ to be $\mathcal{N}(\textbf{N}(\R\ltimes (\R^+)^{\bbN})).$
\end{rem} 

In our setting, marked point processes are constructed in the following way. We start with a stationary point process in $\R$. Let $D_n=X_{n+1}-X_n.$ We mark the point $X_n$ with a sequence of \iid~random variables $K_n=\{L_{n,i}\}_{i\geq 1}$ that model failures as in Section \ref{localtime}. 

We work with the point process under its Palm probability. Let $\lambda=\E[\Phi([0,1]\times (\R^+)^{\bbN})]$ be the intensity of $\Phi$. We assume $0<\lambda<\infty$. 
The Palm probability of $\Phi$ is defined as, for all $A\in \mathcal{N}(\textbf{N}(\R\ltimes (\R^+)^{\bbN}))$,
\begin{align}
\label{sequenceofpointmapprob}
\PP^{0}[A]&=\frac{1}{\lambda |B|} \E\left[\sum_{n\in \Z}\textbf{1}\{X_n\in B\} \textbf{1}\{\Phi\circ \theta_{X_n}\in A\}\right],
\end{align}
for any $B\in \mathcal{B}(\R)$ with positive Lebesgue measure $|B|$, where
\begin{align*}
\Phi\circ \theta_{X_n}=\{(X_m-X_{n},K_{m})\}_{m\in \Z}.
\end{align*}
The probability measure $\PP^0$ can be regarded as the distribution of the process given there is a point at the origin. In fact, $\PP^0[0\in \Phi]=1$. For more on Palm probabilities, see \cite{daley1}, \cite{daley2}, \cite{mecke}, among others. 
 
\subsection{Point-shifts}
 
To provide a unified definition of asymptotic efficiency for sequential \re \break and \che, we resort to the theory of dynamics on point processes induced by point-shifts (for more on the subject, see \cite{BH} and \cite{thorisson1999point}).

Define $\textbf{N}^0(\R^+\ltimes (\R^+)^{\bbN})$ as the subspace of $\textbf{N}(\R^+\ltimes (\R^+)^{\bbN})$ of all counting measures with mass at the origin. Let $\mathcal{N}(\textbf{N}^0(\R\ltimes (\R^+))^{\bbN})$ be the corresponding trace $\sigma-$algebra. 

Let $\theta:\textbf{N}^0(\R^+\ltimes (\R^+)^{\bbN})\to \textbf{N}^0(\R^+\ltimes (\R^+)^{\bbN})$ be the discrete left-shift operator defined by
\begin{align*}
	\theta \psi = \{(X_m-X_{1},K_{m})\}_{m\in \Z},
\end{align*}
with $\theta^n \psi = \{(X_m-X_{n},K_{m})\}_{m\in \Z},~n\in \Z.$ Let $s:\textbf{N}^0(\R\ltimes (\R^+)^{\bbN})\to \R$ be a measurable function such that
\begin{align*} 
	s(\psi)=X_{\alpha_1},~~\hbox{where $\psi(\{X_{\alpha_1},K_{\alpha_1}\})\geq 1$},
\end{align*}
that is, $s$ maps a counting measure to some element of its support. Such a map is called a \emph{point-map}. A point-map $s$ induces a {\em compatible point-shift}, $S$, that maps, in a translation invariant way, every point of a counting measure to another by
\begin{align}
\label{pointshifdef}
S(\psi,X_n)=s(\theta^n \psi)+X_n,	
\end{align}
for all $X_n$ in the support of $\psi$. 
Then, we define the translation by the point-shift $s$, $\theta_s:\textbf{N}^0(\R\ltimes (\R^+)^{\bbN})\to \textbf{N}^0(\R\ltimes (\R^+)^{\bbN})$ as
\begin{align}
\label{translationbyshift}
\theta_s\psi:=\{\psi-X_{\alpha_1}\}=\{(X_m-X_{\alpha_1},K_{m})\}_{m\in \Z}. 	
\end{align}

Inductively, assuming that $s^{n-1}(\psi)$ is defined and letting $\theta^{n-1}_s\psi=\{\psi-s^{n-1}(\psi)\}$, we let $s^n(\psi)=s(\theta^{n-1}_s\psi)+s^{n-1}(\psi)$ and
$\theta^n_s(\psi)=\{\psi-s^n(\psi)\}$. 	

In words, $s$ takes the counting measure and maps it to an element of its support, $X_{\alpha_1}$. Then $\theta_s$ shifts the counting measure so that $X_{\alpha_1}$ is  the origin.  Applying $s$ again to the shifted counting measure, we get some point on the support of $\psi$, say $X_{\alpha_2}$, and $\theta^2_s$ shifts $\psi$ so that $X_{\alpha_2}$ is the origin, and so on. 

\subsection{Ideal times and actual times}

As discussed in Section \ref{localtime}, in both the one task \re and \break \ch, we have the ideal time (when no failures take place) and the actual time (when accounting for failures). In our sequential models, we have an ideal time and an actual time for each iteration. We define these using point-shifts. 

First, let 
\begin{align}
\label{taun}
\tau_n=\inf\{k\geq 1: L_{n,k}>X_{n+1}-X_{n}\}.
\end{align}   
 
The sequential \re point-map is given by $s_R(\Phi)=X_1$, so $S_R(X_n,\Phi)=X_{n+1}$ for all $n$. The translation by this point-map is simply the discrete left-shift operator, i.e., $\theta_{s_R}^n=\theta^n$. For the $n^{th}-$task, the ideal time is $D^R_n=D_n=X_{n+1}-X_{n}$ and the actual time is $T^R_n=\sum_{i=1}^{\tau_n-1}L_{n,i}+D_n.$

The sequential \ch point-map is $s_C(\Phi)=X_{\nu_0}$, where 
\begin{align}
\label{checkpointpointmappoint}
\nu_0=\sup\{k\geq 1: L_{0,\tau_0}\geq X_{k}\}.
\end{align}
 
\begin{figure}
  \centering
      \includegraphics[width=1\textwidth]{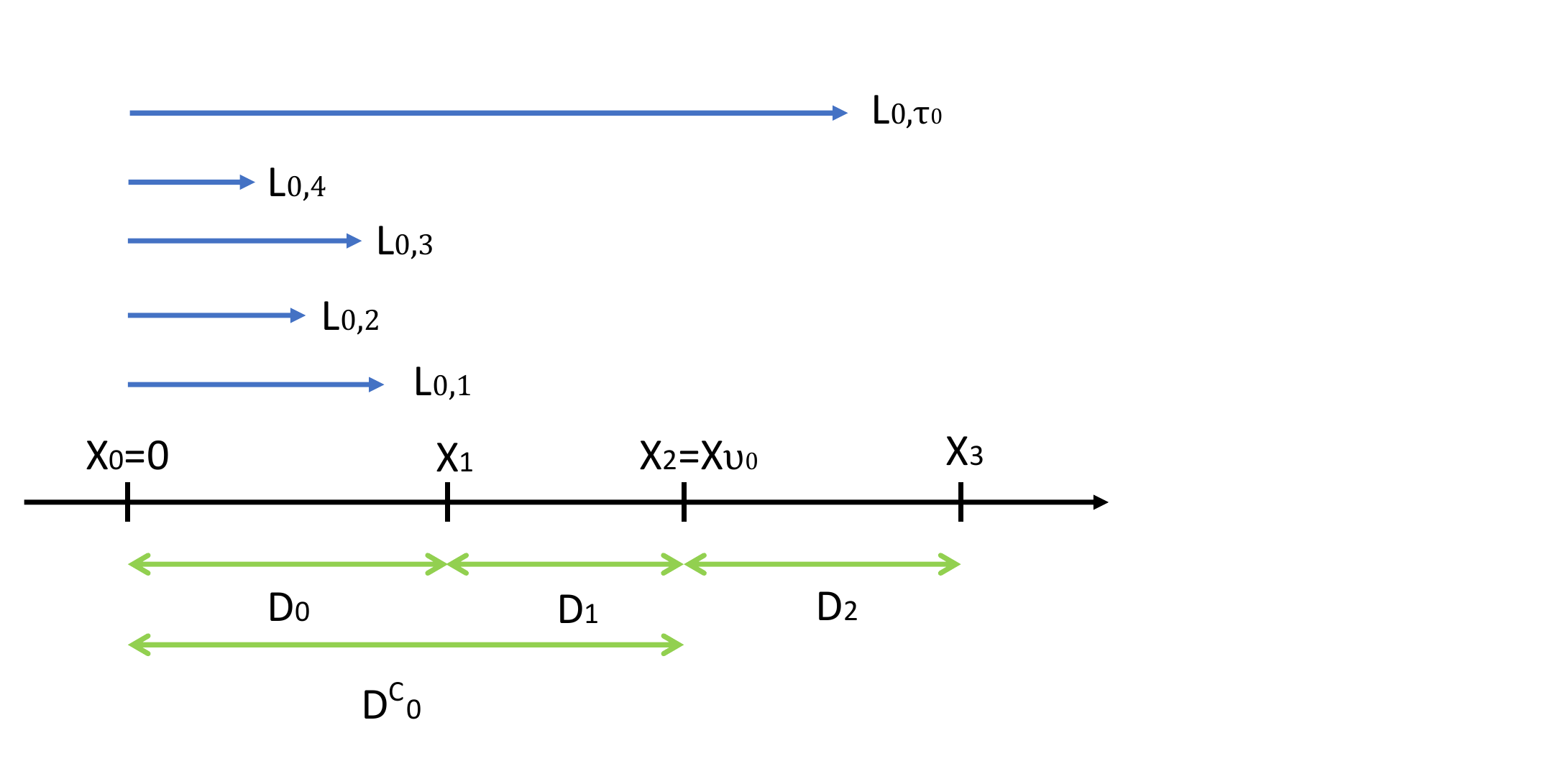}
  \caption{The progress of \ch at the $1^{st}-$iteration. There are four failures before the first checkpoint is surpassed, hence $\tau_0=5$. Then, there are no failures until the system is between checkpoints $X_2$ and $X_3$, so $\nu_1=2$. Here, the first ideal-time is $D^C_0=X_2-X_0$ and the first actual time is $T^C_0=\sum_{i=1}^{5  } L_{0,i}$. }
\end{figure}  

Notice that $\tau_0-1$ is the number of failures before the first checkpoint is surpassed, and $\nu_0$ the index of the next checkpoint secured once the system passes the first one. The $1^{st}-$ideal time is $D^C_0=X_{\nu_0}$ and the $1^{st}-$ actual time is $T^C_0=\sum_{i=1}^{\tau_0} L_{0,i}$. Figure 3 illustrates the first iteration in sequential \che. 
\begin{figure}
  \centering
      \includegraphics[width=1\textwidth]{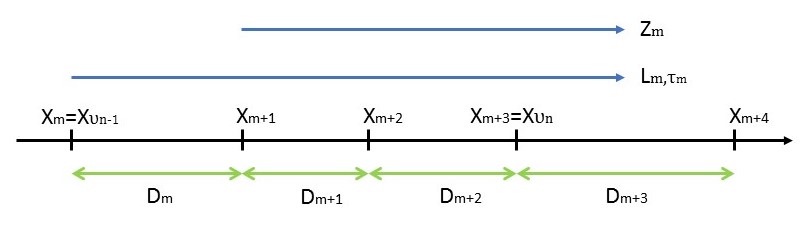}
  \caption{The progress of \ch at the $n^{th}-$iteration.}
\end{figure} 
Now, for each $n$, let 
\begin{align}
\label{definitionofz}
Z_n=L_{n,\tau_n}-D_n.
\end{align}
More generally, as illustrated in Figure 4, at the $n^{th}-$iteration, the $n^{th}-$ideal time is $D^C_n:=X_{\nu_{n}}-X_{\nu_{n-1}}$, where
\begin{align}
\label{nugreekn}
\nu_n=\sup\{k\geq\nu_{n-1}+1: Z_{\nu_{n-1}}>X_k-X_{\nu_{n-1}+1}\},	
\end{align}
with the $n^{th}-$actual time being $T^C_{n}=\sum_{i=1}^{\tau_{\nu_{n-1}}}L_{\nu_{n-1},i}$. We set $\nu_{-1}=0$.

The table below summarizes our notation. 
\begin{center}
 \begin{tabular}{|c|c|c|} 
 \hline
 & \re & \ch \\
 \hline
 Point-map & $s_R(\Phi)=X_1$ & $s_C(\Phi)=X_{\nu_0}$ \\ 
 \hline
 $n^{th}-$ideal time & $D^R_n=X_{n+1}-X_{n}$ & $D^C_n=X_{\nu_n}-X_{\nu_{n-1}}$ \\
 \hline
$n^{th}-$actual time & $ T^R_n=\sum_{i=1}^{\tau_n-1} L_{n,i}+D_n$ & $T^C_{n}=\sum_{i=1}^{\tau_{\nu_{n-1}}}L_{\nu_{n-1},i}$ \\
 \hline
 Point-map translation& $\theta_{s_R}^n=\theta^n$ & $\theta_{s_C}^n=\theta^{\nu_n}$ \\

 \hline
\end{tabular}
\end{center}
\subsection{Asymptotic efficiency}

In this unified framework, we define asymptotic efficiency as the limit ratio of the sum of ideal times to the sum of actual times for both models.
\begin{defi}[Asymptotic Efficiency]
\label{efficiency} 
The asymptotic efficiency is given by, for $i\in \{R,C\}$,
\begin{align}
\label{assymptoticefficiencydefinition}
e(\omega)&=\lim_{N\to \infty} \frac{\sum_{n=0}^{N-1} D^i_n(\omega)}{\sum_{n=0}^{N-1} T^i_n(\omega)}~\PP^0-\hbox{a.s.},
\end{align}
whenever the limit exists. 
\end{defi}
Notice that, when it exists, $0\leq e\leq 1$ $\PP^0-$a.s., as $T^i_n\geq D^i_n$ $n\geq 0$.  

Let $s$ be a point-map. Consider the sequence of probability measures on $(\textbf{N}^0(\R\ltimes (\R^+)^{\bbN}),\mathcal{N}^0(\R\ltimes (\R^+)^{\bbN})$ defined by 
\begin{align}
\label{sequenceofpointshifprobabilities}
\PP^{s,n}=\PP^0\circ (\theta_{s}^n)^{-1}. 
\end{align}Then $\PP^{s,n}$ can be interpreted as the distribution of the point process given that there is a point of the $n^{th}-$image of $S$ at the origin. Suppose $\{\PP^{s,n}\}_{n\geq 0}$ has a weak limit $\PP^{s,\infty}$. As we shall see in detail, when the asymptotic efficiency exists, it is then the ratio of the expectations of $D^i_0$ and $T^i_0$ under $\PP^{s_i,\infty}$, $i\in \{R,C\}$. 

\subsection{General Assumptions}
\label{generalassumptions}

In order to establish the existence of $e$ in sequential \re we assume that the marked point process $\Phi$ is such that, under $\PP^0$, 
\begin{enumerate}
\item the $\hbox{i.i.d.}$ sequence of failure marks, $\{L_{n,i}\}_{i\geq 1}$ is independent of $D_n$ for all $n\geq 0$;
\item  both $D_0$ and $L_{0,1}$ have right-unbounded support;
\item  $\E^0[D_0]$, $\E^0[L_{0,1}]<\infty$.
\end{enumerate}
This set of assumptions allows us to leverage the results of Section \ref{localtime}.

For \ch, besides items $1.$, $2.$, and $3.$ above, we assume that $\Phi$ is a marked renewal process, i.e., under $\PP^0$, $\{D_n\}_{n\in \Z}$ is $\hbox{i.i.d.}$. Moreover, we assume that $\{L_{n,i}\}_{i\geq 1}$ is independent of $D_m$ for all $m\geq n\geq 0$. In words, the failure marks of $X_n$ are independent of the checkpoint intervals ahead. 
  
\section{Sequential \texttt{restart}: main result}
\label{restartmainresults}

Given that the translation by the sequential \re point-map is the discrete left-shift operator $\theta$, the sequence $\{\PP^{s_R,n}\}_{n\geq 0}$ (Equation (\ref{sequenceofpointshifprobabilities})) is constant, with all its elements being equal to $\PP^0$. This result holds as $\theta$ is bijective, so it preserves the Palm measure \cite{t2000coupling}. 

Moreover, from the fact that $\theta$ preserves $\PP^0$, there exists a random variable $L_0$ such that 
\begin{align} 
\label{Ugenerator}
\PP^0[L_0>t]&=\PP^0[L_{n,i}>t],~~~\hbox{$\forall$ $i,n\in \bbN$ and $t\in \R^+$.}
\end{align} 
In the same vein, the sequence $\{D_n\}_{n\geq 0}$ is identically distributed (but not necessarily $\hbox{i.i.d.}$) under $\PP^0$. Consequently, the \re actual time sequence, $\{T^R_n\}_{n\geq 0}$ is also identically distributed under $\PP^0$. 

\begin{thm}
\label{deterministicwalkcor} 
Let $\mathcal{I}$ be the invariant $\sigma-$algebra of $(\PP^0,\theta)$. If $L_0$ has a strict $\PP^0-$tail heavier than $D_{0}$, the asymptotic efficiency exists and it is given by the random variable
\begin{equation}\label{tdwpositivespeed}e=\frac{\E^0[D_0|\mathcal{I}]}{\E^0[T^R_0|\mathcal{I}]},~\PP^0-\hbox{a.s.}.\end{equation}  

If $(\PP,\{\theta_t\}_{t\in \R})$ or, equivalently $(\PP^0,\theta)$ is ergodic, we have $\E^0[D_0|\mathcal{I}]=\E^0[D_0]$ and $\E^0[T^R_0|\mathcal{I}]=\E^0[T^R_0]$, so that the asymptotic efficiency is constant.  In this case, if $\E^0[T^R_0]=\infty$, which holds if $D_0$ does have a $\PP^0-$tail heavier than $L_{0}$, $e=0$ $\PP^0-\hbox{a.s.}$
\end{thm}

\begin{proof}
If $L_0$ has a strict $\PP^0-$tail heavier than $D_{0}$, by Theorem \ref{assymptotics1},
$\E^0[T^R_0]<\infty$, and, as $\E^0[D_0]<\infty$ by assumption, by Birkhoff's Pointwise Ergodic Theorem,
\begin{align*}
e&=\lim_{N\to \infty} \frac{\sum_{n=0}^{N-1} D^i_n(\omega)}{\sum_{n=0}^{N-1} T^i_n(\omega)}=\lim_{N\to \infty} \frac{\sum_{n=0}^{N-1} D_0\circ \theta^n}{\sum_{n=0}^{N-1} T^R_0\circ \theta^n}=\frac{\E^0[D_0|\mathcal{I}]}{\E^0[T^R_0|\mathcal{I}]},~\PP^0-\hbox{a.s}. 
\end{align*}

When $(\PP,\{\theta\}_{t\in \R})$ is ergodic, $\E^0[D_0|\mathcal{I}]$ (resp. $\E^0[T^R_0|\mathcal{I}]$) equals $\E^0[D_0]$ (resp. $\E^0[T^R_0])$
If $\E^0[T^R_0]=\infty$, which holds if $D_0$ does have a $\PP^0-$tail heavier than $L_{0}$, $\E^0[T^R_n]=\infty$ for all $n\geq 0$. Suppose, by contradiction, that 
\begin{align*}
\limsup_{N\to \infty} \frac{\sum_{n=0}^{N-1} D_n}{\sum_{n=0}^{N-1} T^R_n}>\epsilon~\PP^0-\hbox{a.s.}
\end{align*}
for some $\epsilon>0$. It follows that
\begin{align*}
\limsup_{N\to \infty} \frac{1}{N}\sum_{n=0}^{N-1} D_n>\epsilon \limsup_{N\to \infty} \frac{1}{N}\sum_{n=0}^{N-1} T^R_n.
\end{align*}
Let $M>0$ be a fixed integer. Then, as
\begin{align*}
\limsup_{N\to \infty} \frac{1}{N}\sum_{n=0}^{N-1} T^R_n>\limsup_{N\to \infty} \frac{1}{N}\sum_{n=0}^{N-1} \min\{T^R_n,M\}, 
\end{align*}
\begin{align*}
\limsup_{N\to \infty} \frac{1}{N}\sum_{n=0}^{N-1} D_n>\epsilon \limsup_{N\to \infty} \frac{1}{N}\sum_{n=0}^{N-1} \min\{T^R_n,M\}.
\end{align*}
Now $\min\{T^R_n,M\}$ is integrable, so by by Birkhoff's Pointwise Ergodic Theorem, we have
\begin{align}
\label{contradictionargument}
\E^0[D_0]>\epsilon \E^0[\min\{T^R_0,M\}].
\end{align}
Since $\E^0[D_0]<\infty$ and $\E^0[T^R_0]=\infty$, letting $M\to \infty$ on the RHS of (\ref{contradictionargument}) we have a contradiction. 
\end{proof} 

\begin{rem}
Notice that when $(\PP,\{\theta_t\}_{t\in \R})$ is not ergodic, $e$ can be zero with positive probability. Here is a simple example. Consider a stationary marked renewal process constructed in the following way. Let $D_0\sim \exp(\lambda_d)$ under $\PP^0$, and let $c$ be a random variable taking values in $\{0,1\}$ with, $\PP^0[c=0]=\PP^0[c=1]>0$. Then, if $c=0$, $L_0\sim \exp(\lambda_1)$ under $\PP^0$ and, otherwise $L_0\sim \exp(\lambda_2)$. Assume $\lambda_1\geq \lambda_d> \lambda_2$. Then $e=0$ with probability $\PP^0[c=0]$. 
\end{rem}

\begin{rem}
Now since the set 
$$A:=\left\{\lim_{N\to \infty} \frac{\frac{1}{N}\sum_{n=0}^{N-1} D_0\circ \theta^n}{\frac{1}{N}\sum_{n=0}^{N-1} T^R_0\circ \theta^n}=\frac{\E^0[D_0|\mathcal{I}]}{\E^0[T^R_0|\mathcal{I}]}\right\}$$
is strictly $\theta-$invariant, i.e., $\theta A=A$, by property 1.6.1 in \cite{BB1}, $\PP^0[A]=1$ implies $\PP[A]=1$. Therefore, the results above also hold $\PP-\hbox{a.s.}$.
\end{rem} 

\section{Sequential \texttt{checkpointing}: main results}
\label{checkpointsection}

In what follows:
\begin{itemize}
\item  $\Phi$ satisfies the general assumptions for \ch in Section \ref{generalassumptions};
\item $D_{\nu_n}=X_{\nu_n}-X_{\nu_n-1}$, as illustrated in Figure 4, with $\nu_n$ defined in (\ref{nugreekn});
\item $\PP^{s_C,n}:=\PP^0\circ (\theta^{\nu_n})^{-1}$, with $\{\PP^{s_C,n}_+\}_{n\geq 0}$ being the restriction of $\PP^{s_C,n}$ to $\textbf{N}^0(\R^+\ltimes (\R^+)^{\bbN})$.  
\item $\PP^{s_{C},\infty}_+$ denotes the weak limit of the sequence of distributions $\{\PP^{s_C,n}_+\}_{n\geq 0}$, when it exists, and $\E^{s_{C},\infty}_+$ is the expectation operator of $\PP^{s_{C},\infty}_+$. 
\item Assuming $\{D_{\nu_n}\}_{n\ge 0}$ converges weakly under the Palm distribution to a non-degenerate random variable $D_{\infty}$ and letting $\hat{D}_{\infty}$ be an independent random variable distributed like the Palm distribution of $D_{\infty}$, we set
\begin{align}
\label{tauninfty}
\tau_\infty=\inf\{k\geq 1: L_{0,k}>\hat{D}_{\infty}\}
\end{align}
and
\begin{align}
\label{nuinfty}
\nu_\infty=\sup\{k\geq 1: L_{0,\tau_\infty}\geq X_{k}\}.
\end{align}   
\end{itemize}

\begin{rem}
Let $\PP^0_+$ be the restriction of $\PP^0$ to $\textbf{N}^0(\R^+\ltimes (\R^+)^{\bbN})$. Then, under the assumptions of Section \ref{generalassumptions}, $\PP^0_+$ is an independently marked renewal process and, therefore, satisfies the strong Markov property. In this section, all events consider under $\PP^0$ belong to the trace $\sigma-$algebra $\mathcal{N}(\textbf{N}^0(\R^+\ltimes (\R^+)^{\bbN})$. Hence, we keep the notation $\PP^0$ when there is no ambiguity.  
\end{rem}

\begin{thm}
\label{checkpointingmainresult} 
If $\{D_{\nu_n}\}_{n\ge 0}$ converges weakly under its Palm distribution to a non-degenerate random variable $D_{\infty}$, then $\PP^{s_C,\infty}_+$ exists. 

Moreover, if $\E^{0}[D_\infty]$, $\E^0[\nu_\infty]<\infty$, and $L_{0,1}$ does have a strict $\PP^{0}-$tail heavier than $D_{\infty}$, then the asymptotic efficiency exists and it is equal to 
\begin{align*}e=\frac{\E^{s_C,\infty}_+[D^C_0]}{\E^{s_C,\infty}_+[T^C_0]}~~\PP^0-\hbox{a.s.},
\end{align*} 
If $D_{\infty}$ has a $\PP-$heavier tail than $L_{0,1}$, $e=0$ $\PP^0-$a.s.
\end{thm}

For the sake of brevity, in what comes next, we denote the sequence of failure marks $\{L_{n,i}\}_{i\geq 1}$ by $\textbf{L}_{n}.$

\begin{lem}
\label{stoppingtimelemma}
Under $\PP^0$, consider the filtration $\{\mathcal{F}_m\}_{m\geq 0}$ in which 
\begin{align*}
\mathcal{F}_m=\sigma((D_0,\textbf{L}_0),\ldots,(D_m,\textbf{L}_m))~\forall ~m.
\end{align*}  
Then, for all $n\geq 0$, $\nu_n+1$ is a stopping time with respect to $\{\mathcal{F}_m\}_{m\geq 0}$.
\end{lem} 
\begin{proof}
As $\bL_n$ is independent of $\{D_m\}_{m\geq n}$, the stopping time property of $\nu_0+1$ follows from the fact that $\tau_0$ is $\mathcal{F}_0-$measurable, $\nu_0>0$ $\hbox{a.s.}$, and, for all $m \geq 1$,
\begin{align*}
	\{\nu_0+1=m\}=\{X_m\geq L_{0,\tau_0}\}\cap \{X_{m-1}<L_{0,\tau_0}\} \subset \mathcal{F}_{m}.
\end{align*}
Now suppose $\nu_{k}+1$ is a stopping time. Then, for $m\leq k$, $\{\nu_{k+1}+1= m\}=\emptyset$, and, for $m>k$,
\begin{align*}
\{\nu_{k+1}+1= m\}&=\cup_{i=1}^{m-1}(\{\nu_{k}+1=i\}	\cap \{X_{m+1}<X_{i}+Z_{i-1}\leq X_m\}))\subset \mathcal{F}_m, 
\end{align*} 
$Z_{i-1}$ defined in (\ref{definitionofz}). 
\end{proof}

Let $\{\overline{D}_n\}_{n\in \Z}$ be a sequence of $\hbox{i.i.d}$ random variables, independent of $\{D_n\}_{n\in \Z}$ such that $\overline{D}_n$ has the same distribution as $D_0$. Let the total lifetime of the renewal process $\{\overline{D}_n\}_{n\in \Z}$ be 
\begin{align}
\label{betatotallife}
\beta(t)=\overline{D}_n~\hbox{if $X_{n}<t\leq X_{n+1}$}.
\end{align}
 
Given $Z_0=z$, we have $D_{\nu_0}=\beta(z)$. 
Therefore,
\begin{align*}
\PP^0[D_{\nu_0}> x]=\int_0^{\infty} 	\PP^0[\beta(t)>x]f_{Z_{0}}(dt),
\end{align*}
where $f_{Z_0}$ is the distribution of $Z_0$.

The interval $D_{\nu_0}$ tends to be larger than $D_0$, as failures are more likely to happen when checkpoints are more apart. In fact, $D_{\nu_0}$ stochastically dominates $D_0$, as $\PP^0[\beta(t)>x]\geq \PP^0[D_0>x]$ for all $x,t\in \R^+$ \cite{asmussen2003applied}. This is an incarnation of the inspection paradox. Hence, in contrast with \res, $\theta_{s_C}$ does not preserve $\PP^0$ and, consequently,  $\{D^C_n\}_{n\geq 0}$ is not identically distributed under the Palm measure.

A sequence of inter-arrivals $\{\tilde{D}_n\}_{n\geq 0}$ is called a delayed renewal process if $\{\tilde{D}_{n}\}_{n\geq 0}$ is a sequence of independent and non-negative random variables and $\{\tilde{D}_{n}\}_{n\geq 1}$ is $\hbox{i.i.d.}$. In Proposition \ref{delayedrenewalprop1} below, we show that not only $\PP^{s_C,n}_+$ is the distribution of an independently marked delayed renewal process, but also the distribution of the inter-arrivals after the first one is the same under $\PP^{s_C,n}_+$ and $\PP^0$. The result goes along with the interpretation of $\PP^{s_C,n}$ as the distribution of the point process given there is a point of the $n^{th}-$iteration of the point-shift $S_C$ at the origin. To illustrate our case, consider the point process shifted by $\theta^{\nu_0}$. As mentioned above, there is an inspection paradox effect in first interval $D_{\nu_0}$. Nonetheless, as shown below, the inter-arrivals distributions $(X_{\nu_0+2}-X_{\nu_0+1})$, $(X_{\nu_0+3}-X_{\nu_0+2}),\ldots, (X_{\nu_0+j}-X_{\nu_0+j}),\ldots,$ are $\hbox{i.i.d.}$ and have the same distribution under $\PP^0$. This takes place in every iteration: the first inter-arrival interval after the shift is biased and the following ones maintain their distribution, which is that of a typical inter-arrival.

\begin{prop}
\label{delayedrenewalprop1}
For all $n\geq 0$, $\{\PP^{s_C,n}_+\}_{n\geq 1}$ is the distribution of an independently marked delayed renewal process. Moreover, 
\begin{align*}
&\PP^{s_C,n}_+[D_1\in A_1, \textbf{L}_1\in B_1,\ldots,D_m\in A_m, \textbf{L}_m\in B_m,\ldots]\\
&=\PP^0[D_1\in A_1, \textbf{L}_1\in B_1,\ldots,D_m\in A_m, \textbf{L}_m\in B_m,\ldots]
\end{align*}
for all $\{A_i\}_{i\geq 1}\in \mathcal{B}(\R^{+})$ and $\{B_i\}_{i\geq 1}\in \mathcal{B}((\R)^{\bbN})$. 
\end{prop}

\begin{proof}
For $A_0,\ldots,A_j\in \mathcal{B}(\R^+)$ and $B_0,\ldots,B_j\in \mathcal{B}((\R^+)^{\bbN}))$,
\begin{align*}
&\mathbb{P}^{s_C,n}[D_{j}\in A_j,\bL_j\in B_j,\ldots, D_{0}\in A_0,\bL_0\in B_0]\\
&=\mathbb{P}^0[D_{\nu_n+j}\in A_j,\bL_{\nu_n+j}\in B_j,\ldots,D_{\nu_n}\in A_0,\bL_{\nu_n}\in B_0]\\
&=\mathbb{P}^0\left[D_{\nu_n+j}\in A_j,\bL_{\nu_n+j}\in B_j|D_{\nu_n+j-1}\in A_{j-1},\bL_{\nu_n+j-1}\in B_{j-1},\right.\\
&~~~\left.\ldots,D_{\nu_n}\in A_0,\bL_{\nu_n}\in B_0\right]\\
&\times\mathbb{P}^0[D_{\nu_n+j-1}\in A_{j-1},\bL_{\nu_n+j-1}\in B_{j-1}\ldots,D_{\nu_n}\in A_0,\bL_{\nu_n}\in B_0].
\end{align*}

As $\nu_n+1$ is a stopping time, by the strong Markov property of independently marked renewal processes, for every $j>0$,  
\begin{align*}
&\mathbb{P}^0\left[D_{\nu_n+j}\in A_j,\bL_{\nu_n+j}\in B_j|D_{\nu_n+j-1}\in A_{j-1},\bL_{\nu_n+j-1}\in B_{j-1},\right.\\
&~~~\left.\ldots,D_{\nu_n}\in A_0,\bL_{\nu_n}\in B_{j-1}\right]=\PP^0[D_j\in A_j]\PP^0[\bL_j\in B_j].
\end{align*}
By keeping conditioning and applying the strong Markov property:
\begin{align}
\label{propcheckpointrenewaleq2} 
&\mathbb{P}^{s_C,n}[D_{j}\in A_j,\bL_j\in B_j,\ldots, D_{0}\in A_0,\bL_0\in B_0]\nonumber\\
&=\PP^0[D_{\nu_n}\in A_0,\bL_{\nu_n}\in B_0]\prod_{i=1}^{j}\PP^0[D_{0}\in A_i]\PP^0[\bL_0\in B_j])\nonumber\\
&=\PP^0[D_{\nu_n}\in A_0]\PP^0[\bL_{0}\in B_0]\prod_{i=1}^{j}\PP^0[D_{0}\in A_i]\PP^0[\bL_0\in B_j],  
\end{align}
where the last equality follows from independent marking. 
\end{proof}

\begin{cor}
\label{corofconvergence}
If $\{D_{\nu_n}\}_{n\ge 0}$ converges weakly under the Palm distribution to a non-degenerate random variable $D_{\infty}$, then $\{\PP^{s_C,n}_+\}_{n\geq 1}$ converges weakly to a distribution $\PP^{s_C,\infty}_+$. Moreover, $\PP^{s_C,\infty}_+$ is the distribution of an independently marked delayed renewal process in which the first inter-arrival interval is distributed as $D_{\infty}$. 
\end{cor}

\begin{proof}
The results follow from Proposition \ref{delayedrenewalprop1} and taking the limit as $n\to \infty$ in (\ref{propcheckpointrenewaleq2}). 
\end{proof}

\begin{lem}
For all $n\geq 1$,
\begin{align}
\label{distributionofdnun}
\PP^0[D_{\nu_n}> x]=\int_0^{\infty} 	\PP^0[\beta(t)>x]f_{Z_{\nu_{n-1}}}(dt),
\end{align}
where $f_{Z_{\nu_{n-1}}}$ is the distribution of $Z_{\nu_{n-1}}$ under $\PP^0$, with $\beta(t)$ defined in (\ref{betatotallife}).
\end{lem}

\begin{proof}
As  $\PP^0[D_{\nu_n}> x]=\PP^{s_C,{n-1}}_+[D_{\nu_0}>x]$ and $\PP^{s_C,n-1}_+$ is the distribution of a independently delayed renewal process such that $\{(D_i, \textbf{L}_i)\}_{i\geq 1}$ has the same distribution under $\PP^{s_C,n-1}_+$ and $\PP^0$, we have
\begin{align*}
\PP^{s_C,n-1}_+[D_{\nu_0}>x]=\int_0^{\infty} \PP^0[\beta(t)>x]f^{n-1}_{Z_{0}}(dt),
\end{align*}
where $f^{n-1}_{Z_{0}}$ is the distribution of $Z_{0}$ under $\PP^{s_C,n-1}_+$. As $f^{n-1}_{Z_{0}}=f_{Z_{\nu_{n-1}}}$, the result follows. 
\end{proof}

\begin{rem}
So far, we have defined $\{\PP^{s_C,n}_+\}_{n\geq 0}$ and $\PP^{s_C,\infty}_+$ (when it exists) on the space of counting measures. Once it is established that these distributions are concentrated on independently marked delayed renewal processes, we can, without loss of generality, define these measures on the space of discrete sequences in which each term belongs to $\R^+\times (\R^+)^{\bbN}$, equipping it with the standard cylindrical Borel $\sigma-$algebra.  We work on this space in the next proposition. 
\end{rem}

\begin{prop}
\label{limitdistributionmixing}
If $\{D_{\nu_n}\}_{n\geq 0}$ converges weakly  to a non-degenerate random variable, $\theta^{\nu_0}$ preserves $\PP^{s_C,\infty}_+$ and $(\PP^{s_C,\infty}_+,\theta^{\nu_0})$ is mixing. 
\end{prop} 
\begin{proof}
First we show $\theta^{\nu_0}$ preserves $\PP_+^{s_C,\infty}$. Consider the product cylinder set
\begin{align*}
C_{j_0,\ldots,j_l}&:=\left\{(D_n,\bL_{n})_{n\geq 0}: D_{j_0}\in A_{j_0},\textbf{L}_{j_0}\in B_{j_0}, \right. \ldots, \left.D_{j_l}\in A_{j_l}, \textbf{L}_{j_l}\in B_{j_l}\right\},
\end{align*}
where $0\geq j_0>j_1,\ldots> j_l\in \bbN_+$. 
\begin{align*}
&\PP^{s_C,n}_+[\theta^{\nu_0}C_{j_0,\ldots,j_l}]\\
&=\PP^{0}\left[\theta^{\nu_0}\left\{(D_n,\bL_{n})_{n\geq 0}: D_{\nu_n+j_0}\in A_{j_0},\textbf{L}_{\nu_n+j_0}\in B_{j_0}, \right. \ldots,\right.\\
&~~~\left. \left.D_{\nu_n+j_l}\in A_{j_l}, \textbf{L}_{\nu_n+j_l}\in B_{j_l}\right\}\right]\\
&=\PP^{s_C,1}\left[\left\{(D_n,\bL_{n})_{n\geq 0}: D_{\nu_n+j_0}\in A_{j_0},\textbf{L}_{\nu_n+j_0}\in B_{j_0}, \right. \ldots,\right.\\
&~~~\left. \left.D_{\nu_n+j_l}\in A_{j_l}, \textbf{L}_{\nu_n+j_l}\in B_{j_l}\right\}\right]\\
&=\PP^{0}\left[\left\{(D_n,\bL_{n})_{n\geq 0}: D_{\nu_{n+1}+j_0}\in A_{j_0},\textbf{L}_{\nu_{n+1}+j_0}\in B_{j_0}, \right. \ldots,\right.\\
&~~~\left. \left.D_{\nu_{n+1}+j_l}\in A_{j_l}, \textbf{L}_{\nu_{n+1}+j_l}\in B_{j_l}\right\}\right].
\end{align*}
Then, as $\{D_{\nu_n}\}_{n\geq 0}$ converges weakly to $D_{\infty}$, by (\ref{propcheckpointrenewaleq2}), taking the limit as $n\to \infty$ on both sides,
\begin{align}
\label{measurepreservingnu0}
\PP^{s_C,\infty}_+[\theta^{\nu_0}C_{j_0,\ldots,j_l}]&=\PP^{s_C,\infty}_+[C_{j_0,\ldots,j_l}].
\end{align}

By standard extension arguments from product cylinder sets, we conclude that $\theta^{\nu_0}$ preserves $\PP^{s_C,\infty}_+$. 

Next, we prove $(\PP^{s_C,\infty}_+,\theta^{\nu_0})$ is mixing. First, we notice that $D_{\nu_n}$ is  a function of $D_{\nu_{n-1}}$, $\{L_{0,\nu_{n-1}}\}_{i\geq 0}$, and $\{D_n\}_{n\geq \nu_{n-1}+1}$. By independent marking the $\hbox{i.i.d.}$ sequence $\{L_{0,\nu_{n-1}}\}_{i\geq 0}$ is independent of $\nu_{n-1}$ and has the same distribution under $\PP^0$ as $\{L_{0,i}\}_{i\geq 0}$. In the same way, by the strong Markov property the $\hbox{i.i.d.}$ sequence $\{D_n\}_{n\geq \nu_{n-1}+1}$ is independent of $\nu_{n-1}$ and has the same distribution under $\PP^0$ as $\{D_n\}_{i\geq 0}$. Therefore, $\{D_{\nu_n}\}_{n\geq 0}$ is a Markov Chain. 

 Let $C_{j_0,\ldots,j_l}$ and $C_{j'_0},\ldots,C_{j'_q}$ be two product cylinder sets. Following the same steps used to get (\ref{measurepreservingnu0}), for all $m\geq 1$,
\begin{align} 
\label{propergodicthetaseq2}
&\PP^{s_C,\infty}_+[C_{j_0,\ldots,j_l} \cap \theta^m_{s_C}C_{j'_0,\ldots,j'_q}]\nonumber\\
&=\lim_{n\to \infty} \PP^0\left[D_{\nu_n+j_0}\in A_{j_0}, \textbf{L}_{\nu_n+j_0}\in B_{j_0}, \ldots, D_{\nu_n+j_l}\in A_{j_l}, \textbf{L}_{\nu_n+j_l}\in B_{j_l}\cap \right.\nonumber\\
&\left.D_{\nu_{n+m}+j'_0}\in A_{j'_0}, \textbf{L}_{\nu_{n+m}+j'_0}\in B_{j'_0}, \ldots, D_{\nu_{n+m}+j'_q}\in A_{j'_q}, \textbf{L}_{\nu_{n+m}+j'_q}\in B_{j'_q}\right] 
\end{align} 
Then, for all $m$ such that $\nu_{n+m}+j'_0>\nu_{n+1}+j_l$, as $\{D_{\nu_n}\}_{n\geq 0}$ is a Markov chain, (\ref{propergodicthetaseq2}) equals to
\begin{align*}
 &\lim_{n\to \infty} \PP^0\left[D_{\nu_n+j_0}\in A_{j_0}, \textbf{L}_{\nu_n+j_0}\in B_{j_0}, \ldots, D_{\nu_n+j_l}\in A_{j_l}, \textbf{L}_{\nu_n+j_l}\in B_{j_l} \right]\nonumber\\
&~\times \PP^0\left[D_{\nu_{n+m}+j'_0}\in A_{j'_0}, \textbf{L}_{\nu_{n+m}+j'_0}\in B_{j'_0}, \ldots, D_{\nu_{n+m}+j'_q}\in A_{j'_q}, \textbf{L}_{\nu_{n+m}+j'_q}\in B_{j'_q}\right]\\
&=\PP^{s_C,\infty}_+[C_{j_0,\ldots,j_l}]\PP^{s_C,\infty}_+[C_{j'_0,\ldots,j'_q}].
\end{align*}
Again, invoking standard approximations arguments, we conclude that for all $C,C'\in \mathcal{B}((\R^+\times (\R^+)^{\bbN})^{\bbN})$, 
$\lim_{m\to \infty} \PP^{s_C,\infty}_+[C\cap \theta^{\nu_n}C']=\PP^{s_C,\infty}_+[C]\PP^{s_C,\infty}_+[C']$,
completing the proof. 
\end{proof}

\begin{lem}
\label{equalitywhenmeasure1}
Suppose $\{D_{\nu_n}\}_{n\geq 0}$ converges weakly to a non-degenerate random variable. Let $A$ be a strictly $\theta^{\nu_0}-$invariant event in $\mathcal{N}(\textbf{N}^0(\R^+\ltimes (\R^+)^{\bbN})$. If $\PP^{s_C,\infty}_+[A]=1$, then $\PP^0[A]=1$.  
\end{lem}

\begin{proof}
By Corollary \ref{corofconvergence}, as $A$ is $\theta^{\nu_0}-$invariant, i.e., $\theta^{\nu_0}A=A$, and, hence, for all $n$,  $\theta^{\nu_n}A=A$,
\begin{align*}
1=\PP^{s_C,\infty}_+[A]&=\lim_{n\to \infty} \PP^0[\theta^{\nu_n}A]=\lim_{n\to \infty} \PP^0[A]=\PP^0[A].
\end{align*}
\end{proof} 

\begin{proof}[Proof of Theorem \ref{checkpointingmainresult}]
As $\{D_{\nu_n}\}_{n\ge 0}$ converges in distribution to a \break non-degenerate random variable, by Corollary \ref{corofconvergence}, $\PP^{s_C,\infty}_+$ exists. Moreover, by Proposition \ref{limitdistributionmixing},  $(\PP^{s_C,\infty}_+,\theta^{\nu_0})$ is mixing. Then, by Birkhoff's pointwise ergodic theorem, for any measurable function $h: \textbf{N}^0(\R^+\times (\R^+)^{\bbN})\to \R^+$ such that $h\in \mathcal{L}^1(\PP^{s_C,\infty}_+)$, 
\begin{align*}
\lim_{n\to \infty} \frac{1}{N} \sum_{n=0}^{N-1} h\circ \theta^{\nu_n}= \E^{s_C,\infty}_+[h],~\PP^{s_C,\infty}_+-\hbox{a.s}.	
\end{align*}

Next, notice that
\begin{align*}
\E^{s_C,\infty}_+[D_C^0]&=\E^{s_C,\infty}_+\left[\sum_{n=0}^{\nu_0-1}D_n\right]\\
&=
\E^0[D_{\infty}\textbf{1}\{\nu_\infty=1\}]+\E^0\left[\sum_{n=1}^{\nu_\infty-1}D_n \textbf{1}\{\nu_\infty> 1\}\right].
\end{align*} 
Now, as $\E^0[D_{\infty}]<\infty$, $\E^0[D_{\infty}\textbf{1}\{\nu_\infty=1\}]<\infty$.  Moreover, as $\{D_n\}_{n\geq 0}$ is $\hbox{i.i.d.}$ under $\PP^0$,
\begin{align*}
\E^0\left[\sum_{n=1}^{\nu_\infty-1}D_n \textbf{1}\{\nu_\infty\geq 1\}\right]\leq\E^0\left[\sum_{n=0}^{\nu_\infty}D_n \right]\leq \E^0\left[\sum_{n=0}^{\nu_\infty+1}D_n \right].
\end{align*}
Following the same reasoning as in Lemma \ref{stoppingtimelemma}, $\nu_{\infty}+1$ is a stopping time with respect to the natural filtration of $\{(D_n,\textbf{L}_n)\}_{n\geq 0}$. Hence, by the general version of Wald's equality for stopping times, 
\begin{align*}
\E^0\left[\sum_{n=0}^{\nu_\infty+1}D_n \right]=\E^0[\nu_{\infty}+1]\E^0[D_0],
\end{align*}
which is finite as $\E^0[\nu_{\infty}]<\infty$ by assumption.

It follows that $\E^{s_C,\infty}_+[D_C^0]<\infty$. Hence, by Birkhoff's pointwise ergodic theorem,
\begin{align}
\label{checkpointingeq1}
\lim_{n\to \infty} \frac{1}{N} \sum_{n=0}^{N-1} D_0^C\circ \theta^{\nu_n}= \E^{s_C,\infty}_+[D_C^0],~\PP^{s_C,\infty}_+-\hbox{a.s}.	
\end{align}

As $\E^{s_C,\infty}_+[T_0^C]=\E^0[\sum_{i=1}^{\tau_{\infty}} L_{0,i}]$, if we assume that $L_{0,1}$ has a strict $\PP^{0}-$tail heavier than $D_{\infty}$, $T^C_0 \in \mathcal{L}^1(\PP^{s_{C},\infty}_+)$, so 
\begin{align}
\label{checkpointingeq2}
\lim_{n\to \infty} \frac{1}{N} \sum_{n=0}^{N-1} T_0^C\circ \theta^{\nu_n}= \E^{s_C,\infty}_+[T_0^C],~\PP^{s_C,\infty}_+-\hbox{a.s}.	
\end{align}
Therefore, by (\ref{checkpointingeq1}) and (\ref{checkpointingeq2}),
\begin{align}
\label{asymptoticcheckpointefficiency}
e:=&\lim_{N\to \infty} \frac{\sum_{n=0}^{N-1} D^C_n}{\sum_{n=0}^{N-1} T^C_n}=\lim_{N\to \infty} \frac{\sum_{n=0}^{N-1} D^C_0\circ \theta^{\nu_n}}{\sum_{n=0}^{N-1} T^C_0\circ \theta^{\nu_n}}= \frac{\E^{s_C,\infty}_+[D^C_0]}{\E^{s_C,\infty}_+[T^C_0]}.
\end{align}
Let 
\begin{align*}
A:=\left\{\lim_{N\to \infty} \frac{\sum_{n=0}^{N-1}D^C_0\circ \theta^{\nu_n}}{\sum_{n=0}^{N-1}T^C_0\circ \theta^{\nu_n}}=\frac{\E^{s_C,\infty}_+[D^C_0]}{\E^{s_C,\infty}_+[T^C_0]}\right\}. \end{align*}
Since
\begin{align*}
\left(\lim_{N\to \infty} \frac{\sum_{n=0}^{N-1}D^C_0\circ \theta^{\nu_n}}{\sum_{n=0}^{N-1}T^C_0\circ \theta^{\nu_n}}\right)\circ \theta^{\nu_0}=\lim_{N\to \infty} \frac{\sum_{n=1}^{N-1}D^C_0\circ \theta^{\nu_n}}{\sum_{n=1}^{N-1}T^C_0\circ \theta^{\nu_n}}
\end{align*}
and $\E_+^{s_C}$ is preserved by $\theta^{\nu_0}$, $A$ is strictly $\theta^{\nu_0}-$invariant. Then, by Lemma \ref{equalitywhenmeasure1}, (\ref{asymptoticcheckpointefficiency}) holds $\PP^0$-a.s., as desired. 

If $D_\infty$ does have a $\PP^{0}-$tail heavier than $L_{0,1}$, $\E^{s_C,\infty}_+[T^C_0]=\infty$. Then, the same argument used in Theorem \ref{deterministicwalkcor} applies to show $e=0$. 
\end{proof}

In the rest of this section, we focus on the case of exponential failure marks, in which we can say more about $\PP^{s_C,\infty}_+$. 

\subsection{Exponential failures and universal checkpoints}

Suppose that $\PP^0[L_{0,0}\leq x]=1-e^{-\lambda x}$, $\lambda>0$, i.e., the failure marks are exponentially distributed. We show the conditions of Theorem \ref{checkpointingmainresult} are then satisfied if $\E^0[\nu_1]<\infty$, so $e$ exists. We also show there is a sequence of checkpoints that will be activated regardless of the initial checkpoint from which we start the system. These are called \emph{universal checkpoints}. 

\begin{thm}
\label{exponentialfailurethm1}
	Suppose that $\Phi$ is a independently marked renewal process with exponentially distributed failure marks. Moreover, assume that $\E^0[D_{\nu_1}],~\E^0[\nu_1]<\infty$. Then $e$ is well-defined. 
\end{thm}

\begin{proof}
In order to apply Theorem \ref{checkpointingmainresult}, we need to show that $\PP^{s_C,\infty}_+$ exists, $\E^0[D_{\infty}]$, and $\E^0[\nu_\infty]<\infty$. Due to the memoryless property of the exponential distribution, the sequence of random variables $\{D_{\nu_n}\}_{n\geq 1}$ is identically distributed, and thus converges in distribution. To see this, notice that $Z_n$ defined in (\ref{definitionofz}) is exponentially distributed with parameter $\lambda$ for all $n$. Consequently, from (\ref{distributionofdnun}),
\begin{align*}
\PP^0[D_{\nu_n}>x]=\int_0^{\infty} \PP^0[\beta(t)>x]\lambda e^{-\lambda t} dt,~\forall~n\ge 0.
\end{align*}
By same reasoning, $\nu_{\infty}$ has the same distribution under Palm as $\nu_j$ for $j\geq 1$, so $\E^0[\nu_1]=\E^0[\nu_{\infty}]$. 
\end{proof}

From (\ref{pointshifdef}) we have $s_C(\Phi,X_n)=X_{\kappa_n}$, where
\begin{align}
\label{kappan}
\kappa_n&=\sup\{k\geq n: L_{n,\tau_n}\geq X_{k}\},
\end{align}
where $\tau_n$ as in (\ref{definitionofz}). 
Then $S^2_c(\Phi,X_n)=S_C(\Phi,X_{\kappa_n})$, and so on. Let $\mathcal{H}_n=\{S^j_c(\Phi,X_n)\}_{j\geq 1}$. In words, $\mathcal{H}_n$ corresponds to the sequence of checkpoints if the system starts at $X_n$. Notice that $\mathcal{H}_0=\{X_{\nu_n}\}_{n\geq 1}.$  

Accordingly, we call $\mathcal{H}_n$ the set of \emph{active checkpoints} of $X_n$ or the \emph{checkpoint trajectory} of $X_n$. We also say that if $X_m\in \mathcal{H}_n$, then $X_m$ is \emph{activated} by $X_n$.   By convention, $X_n\notin \mathcal{H}_n$.  We allow $n$ to be negative. For example, the process may start from $X_{-2}$. 

\begin{defi}[Universal Checkpoints]
Suppose $X_m$ is such that for all $k<m$ there exists $j\geq 1$ such that $S^j_c(\Phi,X_k)=X_m$ or, equivalently for all $k<m$, $X_m\in \mathcal{H}_k$. Then  $X_m$ is a universal checkpoint. 
\end{defi}
As the name suggests, universal checkpoints, if they exist, are activated if we start the system from any checkpoint that precedes them. 

\begin{thm}
\label{exponentialfailurethm1}
	Suppose that $\Phi$ is a independently marked renewal process with exponentially distributed failure marks. Moreover, assume $\E^0[\nu_1]<\infty$. Then there exists a sequence of universal checkpoints. 
\end{thm}

 Let
\begin{align*}
	N_n=\#\{m\in \Z:m<n~\hbox{and}~S(\Phi,X_m)>X_n\}, 
\end{align*}
and notice that if $N_n=0$, then $X_n$ is a universal checkpoint. Then, Theorem \ref{exponentialfailurethm1} follows directly from the proposition below. 
\begin{prop}
\label{acessorypropositionexponential}
Under the assumptions of Theorem \ref{exponentialfailurethm1}, for all $k\ge 0$, the process $\{N_n\}_{n\in \Z}$ admits a subsequence $\{N_{n_l}\}_{l\in \Z}$ such that 
$N_{n_l}=k$.
\end{prop}

\begin{proof}
Let $\kappa_n$ as in (\ref{kappan}). Since $\Phi$ is an indepedently marked renewal process, the sequence $\{\kappa_n\}_{n\in \Z}$ is identically distributed. In particular $\kappa_n$ has the same distribution as $\nu_1=\kappa_0$

Define the events $A_n=\{S(\Phi,X_{-n})>0\}$. Then
	\begin{align*}
	\sum_{n=1}^{\infty} \PP^0[A_n]&=\sum_{n=1}^{\infty} \PP^0[\kappa_{-n}>n]=\sum_{n=1}^{\infty} \PP^0[\nu_1>n]=\E^0[\nu_1]<\infty. 	
	\end{align*}
Therefore, by the Borel-Cantelli Lemma, $\PP^0[A_n~\hbox{i.o.}]=0$. Hence $N_0<\infty$ $\hbox{a.s.}$. By the same reasoning, we conclude that $N_n<\infty$ $\hbox{a.s.}$ for all $n$. 

Thanks to the memoryless property of the exponential distribution, the value of $N_n$ only depends on $N_{n-1}$, i.e., $\{N_n\}_{n\in \Z}$ is a Markov Chain. Now suppose $N_{n-1}=k$. That means that there are $k$ checkpoint trajectories that do not use $X_{n-1}$ as a checkpoint. Given $D_n=t$, a trajectory that does not use $X_{n-1}$ as a checkpoint will not activate $X_n$ with probability 
$e^{-\lambda t}$, and will activate $X_n$ with probability $1-e^{-\lambda t}$. 

Considering the checkpoint trajectory of $X_{n-1}$ we have that, for $1\leq j\leq k$, 
\begin{align*}
\label{probabilitymatrix} 
&\PP^0[N_n=j|N_{n-1}=k]\\
&~~=\PP^0[N_n=j|N_{n-1}=k~\hbox{and}~X_{n-1}\in \mathcal{H}_n]\PP^0[X_{n-1}\in \mathcal{H}_n]\\
&~~+\PP^0[N_n=j|N_{n-1}=k~\hbox{and}~X_{n-1}\notin \mathcal{H}_n]\PP^0[X_{n-1}\notin \mathcal{H}_n]\\
&~~=\left(\int_0^{\infty}\left(\begin{array}{c} k \\ j \end{array}\right)(e^{-\lambda t})^j(1-e^{-\lambda t})^{k-j}f_D(dt)\right)\left(\int_0^{\infty}(1-e^{-\lambda t})f_D(dt)\right)\\
&~~+\left(\int_0^{\infty}\left(\begin{array}{c} k \\ j-1 \end{array}\right)(e^{-\lambda t})^{j-1}(1-e^{-\lambda t})^{k-j+1}f_D(dt)\right)\left(\int_0^{\infty}e^{-\lambda t}f_D(dt)\right).
\end{align*}
Moreover,
\begin{align*} 
&\PP^0[N_n=k+1|N_{n-1}=k]=\left(\int_0^{\infty}(e^{-\lambda t})^kf_D(dt)\right)\left(\int_0^{\infty}e^{-\lambda t}f_D(dt)\right),\\
&\PP^0[N_n=0|N_{n-1}=k]\\
&=\left(\int_0^{\infty}(1-e^{-\lambda t})^kf_D(dt)\right)\left(\int_0^{\infty}(1-e^{-\lambda t})f_D(dt)\right).
\end{align*}
Since $\{N_n\}_{n\in \Z}$ is a sequence of marks of the stationary ergodic marked point process $\Phi$, if $\PP^0[N_0=k]>0$, there exists a sequence $\{n_l\}_{l\in \Z}$ such that $N_{n_l}=k$. Since $N_n<\infty$, for all $n$, by the computations above, $\PP^0[N_0=k]>0$ for all $k\geq 0$.
\end{proof}

\section{Extensions}
\label{restartdiscussion}
 
\subsection{Markov renewal processes}
\label{mrpsection}

Markov Renewal Processes give us instances in which the asymptotic efficiency exists, even though the process does not start is not at steady state when tasks start being executed. For simplicity, here we work out sequential \re case, although the same results can be achieved within the sequential \ch with exponential failure marks.  First, we develop the results of a system that admits a steady state but the initial distribution is arbitrary. After that, we indicate how one can fit the Markov Renewal structure in the setting of Theorem \ref{deterministicwalkcor}. 

Here we consider different states for the ideal task times and failures marks, which are are driven by a Markov Chain whose state space is countable. We briefly review the Markov Renewal process structure, adapting it to account for failure marks. For a more complete treatment of the subject see \cite{daley2} and \cite{asmussen2003applied}. In short, the distribution of the task sizes and failure marks depends on the current state and the next state to be visited of a Markov Chain. 

Consider a Markov Chain $\{Y_n\}_{n\geq 0}$ whose space state $S$ is countable. Let $\textbf{A}=(a_i)_{i\in S}$ be its initial distribution and $\textbf{P}=(p_{ij})_{ij\in S}$ its transition matrix. 

For $x,y\in \R^+$, let $G_{0,i,j}(x,y)$ and $G_{i,j}(x,y)$ be two joint distributions on $\R^2$. We then consider the trivariate sequence $\{(Y_n,D_n,L_n)\}_{n\in \bbN}$ defined in a probability space $(\Omega,\mathcal{F},\PP)$ such that
\begin{enumerate}
\item $\PP[Y_0=k_0]=a_{k_0}$;
\item $\PP[D_0\leq x,L_0\leq y,Y_1=k_1|Y_0=k_0]=G_{0,k_0,k_1}(x,y)p_{k_0k_1}$;
\item For $n\geq 1$:
\begin{align*}
	&\PP[D_n\leq x,L_n\leq y,Y_n=k_n|Y_0=k_0,D_0\leq x_0,L_0\leq y_0,\ldots,\\
	&~Y_n=k_n,D_{n-1}\leq x_{n-1},L_{n-1}\leq y_{n-1}]\\
	&=\PP[D_n\leq x,L_{n}\leq y,Y_n=k_n|Y_{n-1}=k_{n-1}]=G_{k_{n-1}k_n}(x,y)p_{k_{n-1}k_n}.
\end{align*}
\end{enumerate}

We further assume that $G_{0,i,j}(x,y)=F^D_{0,i,j}(x)F^{L}_{0,i,j}(y)$ and $G_{i,j}(x,y)=F^D_{i,j}(x)F^{L}_{i,j}(y)$, namely, conditional on $(Y_n,Y_{n-1})$, $L_{n}$ is independent of $D_n$.
Notice that letting $y\to \infty$ we have the classical Markov Renewal Process \cite{janssen2006applied}. 

\begin{assum}
We impose conditions so that $X_n\to \infty$ as $n\to \infty$ $\PP-$a.s.: (i) $\textbf{P}$ is irreducible, (ii) there exists a non-trivial probability measure $(\pi_i)_{i\in S}$ such that $\pi P=P$ and $\sum_{i\in S} \pi_i \mu_i<\infty$, where
\begin{align*}
\mu_{i}=\sum_{j\in S} p_{ij} \int_0^{\infty} x F^{D}_{ij}(dx). 	
\end{align*}
\end{assum}

We write $\hat{D}_{ij}$ (respectively $\hat{L}_{ij}$) the random variable corresponding to the task size (failure mark) conditional on $Y_n=i$ and $Y_{n+1}=j$. It follows that $\hat{D}_{ij}$ (resp. $\hat{L}_{ij}$)  has distribution $F^D_{i,j}(x)$ (resp. $F^L_{i,j}(x)$). 

\begin{assum}
\label{MRPassumption}
For all $i,j\in S$, $\hat{D}_{ij}$ and $\hat{L}_{ij}$ are integrable random variables with right-unbounded support. Moreover, $\hat{D}_{ij}$ is independent of $\hat{L}_{ij}$. 
\end{assum}

\begin{assum}
The embedded Markov Chain $\{Y_n\}_{n\in \bbN}$ is ergodic, with unique stationary distribution given by $\{\pi_i\}_{i\in S}$. 	
\end{assum}

Then, we proceed to derive expressions for $\E[D_0]$ and $\E[T^R_0]$ when the chain is in steady state. First, for each pair of states $(i,j)$ one can compute $\E[\hat{T}_{ij}^R]$, the expected value of the actual time $\hat{T}_{ij}^R$, when $Y_n=i$ and $Y_{n+1}=j$, as in Theorem \ref{assymptotics1}, with $L_0=\hat{L}_{ij}$ and $D=\hat{D}_{ij}$. By unconditioning, 
\begin{align*}
\E[\hat{D}_{i}]=\sum_{j\in S}p_{ij}\E[\hat{D}_{ij}]~\hbox{and}~\E[T^R_{i}]=\sum_{j\in S} p_{ij}\E[\hat{T}_{ij}^R],
\end{align*}
where $\E[\hat{D}_{i}]$ (resp. $\E[T_{i}^R]$) is the expected size of the task (resp. actual time) when the chain is at state $i$.
\begin{rem}
 A pair of states $(i,j)$ is called \emph{slow} if $\hat{D}_{ij}$ has a $\PP-$tail heavier than $\hat{L}_{ij}$. It follows from Assumption \ref{MRPassumption} if that $(i,j)$ is a slow pair of states then $\E[T_{i}^R]=\infty$. 
\end{rem}
 
Due to the Strong Law of Large Numbers for functionals of a Markov renewal process \cite{janssen2006applied},
\begin{align*}
\lim_{n\to \infty} \frac{1}{N} \sum_{n=0}^{N-1} D_n = \sum_{i\in S} \pi_i \E[\hat{D}_i], 	
\end{align*}
and, if $\E^0[T_{i}^R]<\infty$ 
\begin{align*}
\lim_{n\to \infty} \frac{1}{N} \sum_{n=0}^{N-1} T^R_n = \sum_{i\in S} \pi_i \E[T_{i}^R]. 	
\end{align*}

Hence, if there are no slow states:
\begin{align*}
e=\frac{\sum_{i\in S} \pi_i \E[\hat{D}_i]}{\sum_{i\in S} \pi_i \E[T_{i}^R]}.	
\end{align*}

The next example shows that $e$ can be equal to $0$ even in the absence of slow states. 
\begin{ex}
Suppose $\hat{D}_{ij}\sim \exp\left(p_{ij}2^{j}\pi_i2^i\right)$ and $\hat{L}_{ij}\sim \exp\left(p_{ij}2^{j}\pi_i(2^{i}-1)\right)$ for all $i\in \bbN=S$. We assume $0<\pi_i<1$ for all $i\geq 1$. Notice that each pair of states is not slow. However, $\sum_{i=1}^{\infty} \pi_i \E[\hat{D}_i]=1$ and $\sum_{i\in S} \pi_i \E[T_{i}^R]=\sum_{i=1}^{\infty} 1 = \infty$, so $e=0$. This holds as, if $\hat{D}_{ij}\sim \exp\left(\beta_{ij}\right)$ and $\hat{L}_{ij}\sim \exp\left(\alpha_{ij}\right)$ we have, by Theorem \ref{assymptotics1},
\begin{align*}
\E[T_{ij}^R]&=\frac{1}{(\beta_{ij}-\alpha_{ij})}	.
\end{align*}
\end{ex}

We discuss briefly how Markov Renewal Process can be fit into the framework of Theorem \ref{deterministicwalkcor}. Define a point process $\hat{\Phi}$ by $X_0=0$ and $D_n=X_{n+1}-X_n$ and a Semi-Markov Process $\{Y(t)\}_{t\in \R}$ by $Y(t)=Y_n,~\hbox{when}~X_n\leq t <X_{n+1}.$

Assume that $\sum_{i\in S} \pi_i \E[D_i]<\infty$  and $\hat{\Phi}([0,t))<\infty$ for all $t$. Then let $\PP^0=\PP\circ Y^{-1}$. 
Then there exists a probability space endowed with a flow, $(\hat{\PP},\Omega,\mathcal{F},\{\theta\}_t)$, such that $\hat{\PP}$ is $\theta_t-$invariant, $\hat{\Phi}$ is a stationary point process whose Palm distribution is $\PP^0$ (\cite{BB1}, Chapter 1). 

Hence, by marking this process with the failure marks $\{L_{n,i}\}_{i\geq 1}$,  Theorem \ref{deterministicwalkcor} holds with $\E^0[D_0]=\sum_{i\in S} \pi_i \E[\hat{D}_i]$ and $\E^0[T^R_0]=\sum_{i\in S} \pi_i \E[T_{i}^R]$. 

\subsection{Repetition of tasks based on a Random Walk}
\label{prandomwalk}

We now proceed to study the asymptotic efficiency of \re when the tasks to be completed follows a transient simple random walk, i.e., there is a chance $p<\frac{1}{2}$ that, once a task is completed, progress is lost and the system returns to the previous task. For example, after completing task $D_m$, progress might be lost and the system resumes from task $D_{m-1}$, which is again subject to failures. We  assume that this extra failure source is independent of the failure marks and ideal times and $\Phi$ is a renewal process. We show that the asymptotic efficiency exists and we provide a lower bound for it.

\begin{rem}
In order to simplify the proofs and, without loss of generality, we do not model any sort of \emph{boundary effect}, i.e., the tasks $D_{-1}=X_0-X_{-1}$, $D_{-2}=X_{-1}-X_{-2}$ and so on are well-defined.	
\end{rem}

Let $\{\xi_n\}_{n\geq 1}$ be an $\hbox{i.i.d.}$ sequence of Bernoulli random variables such that $\PP[\xi_n=-1]=p$ and $\zeta_n=\sum_{i=1}^{n} \xi_i$ (with $\xi_0=0$).
At the $n^{th}-$iteration, the task being completed is $D^W_n=D_0\circ \theta^{\zeta_{n-1}}$.

We assume that the failure marks are such that $L_{m,0}$ is independent of $L_{n,0}$ for all $n,m\in \Z$ and are re-drawn if a task is repeated.  Under this assumption, instead of marking each point of the process with a sequence of $\hbox{i.i.d.}$ random variables $\{L_{n,i}\}_{i\geq 1}$, we can simplify matters by considering a single sequence of $\hbox{i.i.d.}$ random variables $\{L_{i}\}_{i\geq 1}$, such that $L_{i}$ has the same distribution as $L_0$ under $\PP^0$ to compute the actual time spent on task $D^W_n$. 
Then let $\tau^W_0:=\inf\{i\geq 1: L_i>D^W_0\}$ and $
\tau^W_n:=\inf\{i> \tau^W_{n-1}: L_{i}>D^W_n\}$. So, the actual time spent on task $D^W_n$ is given by $T_{n}^W=\sum_{i=\tau^W_{n-1}}^{\tau^W_n+1} L_{i}+D_n^W.$

Let $\vartheta_n=\inf\{j\geq 1: \zeta_j=n\}$, i.e., $\{\vartheta_n\}_{n\geq 0}$ is the sequence of ladder epochs of the random walk $\zeta_n$ with $\vartheta_0=0$ \cite{feller2}. It follows that 
$D_n=X_{n+1}-X_n=X_{\zeta_{\vartheta_{n+1}}}-X_{\zeta_{\vartheta_{n}}}.$

Given those definitions, one can regard $T^R_n=\sum_{j=\vartheta_{n-1}}^{\vartheta_n-1}T^W_j $ as the total actual time necessary to complete the task $D_n$. As before, the asymptotic efficiency is
\begin{align*} 
e_p=\lim_{N\to \infty}\frac{\sum_{n=0}^{N-1} D_n}{\sum_{n=0}^{N-1} T_{n}^R},~~\PP^0-\hbox{a.s.}, 
\end{align*}
whenever the limit exists. 

Using the fact that $\Phi$ is an independently marked renewal process, $\E^0[T^W_{j}]=\E^0[T_{l}^W]$ for all $l,j\geq 0$, as the Palm expectation of $T_{j}^W$ only depends on the distribution of $L_0$ and $D_0$ under $\mathbb{P}^0$. 

\begin{prop}
\label{propspeedrw} 
Suppose $L_0$ has a strict $\PP^0-$tail heavier than $D_0$. Then 
\begin{align*}
e\geq e_p= \frac{\gamma}{\rho} e,
\end{align*}
where $e$ is the asymptotic efficiency when $p=0$, $\rho$ is the probability that the random walk $\{\zeta_n\}_{n\geq 0}$ never returns to zero and $\gamma$ is the probability that the walk never goes below zero. On the other hand, if $D_0$ has a $\PP^0-$tail heavier than $L_0$, $e_p=0$. 
\end{prop} 

To prove Proposition \ref{propspeedrw} we rely on the next lemmas, following a similar reasoning as in \cite{peres2}.  

We say that $\upsilon$ is a regeneration epoch if it is a ladder epoch and, moreover, $\zeta_{j}\geq \zeta_{\upsilon}$ for all $j> \upsilon$. That is, a regeneration epoch takes place when the walk reaches a certain level for the first time and never returns below it. 

Lemma \ref{lem1speedrw} below is proved in a more general context, namely, for the nearest-neighbor random walk on $\Z$ with site-dependent transition probabilities \cite{dembo1996tail}. Lemma \ref{lem2speedrw} is a classical result on transient random walks whose proof can be found in \cite{spitzer2013principles}. 

\begin{lem}
\label{lem1speedrw}
If $p<\frac{1}{2}$ there exists a sequence $\{\upsilon_m\}_{m\geq 1}$ of regeneration epochs. The sequence $\{\upsilon_{n+1}-\upsilon_n\}_{n\geq 1}$ is $\hbox{i.i.d.}$ as well as the sequence $\{\zeta_{\upsilon_{n+1}}-\zeta_{\upsilon_n}\}_{n\geq 1}$.
\end{lem}

\begin{lem}
\label{lem2speedrw}
Let $R_n$ be the number of distinct sites visited by the walk $\{\zeta_k\}_{k\geq 0}$ after $n$ steps. Then 
\begin{equation}
\lim_{n\to \infty}\frac{\E^0[R_n]}{n}=\rho,
\end{equation}
where $\rho$ is the expected number of visits of the random walk to $0$. 
\end{lem}
\begin{lem}
\label{lem3speedrw}
$\E^0[\upsilon_{1}-\upsilon_0]<\infty$
\end{lem}

\begin{proof}
Fix any $k$ and assume $\zeta_k=m$. By definition 
\begin{align*}
&\{\hbox{$\zeta_k$ is a regeneration epoch}\}\\
&=\{\forall~l<k,~\zeta_l<m~\hbox{and}~\zeta_k=m\}\cap \{\zeta_l>m,~\forall~l>k\}\\
&=\{\zeta_k~\hbox{is a ladder epoch}\}\cap \{\zeta_l>m~\forall~l>k\}. 
\end{align*}
Then, by the Markov property, and noticing that $\gamma=\PP^0[\{\zeta_l>m~\forall~l>k\}]$
\begin{align*}
	&\PP^0[\{\hbox{$\zeta_k$ is a regeneration epoch}\}]\\
	&=\PP^0[\{\zeta_k~\hbox{is a ladder epoch}\}]\PP^0[\{\zeta_l>m,~\forall~l>k\}]\\
	&=\gamma \PP^0[\{\zeta_k~\hbox{is a ladder epoch}\}].
\end{align*}

Now let $\hat{R}_n$ be the number of regeneration epochs on the interval $[1,n]$, i.e.,
\begin{align}
\label{eq1lemma3rw}	
\hat{R}_n&=\sum_{k=1}^n \ind\{\hbox{$\zeta_k$ is a regeneration epoch}\}\Rightarrow\nonumber\\	
\E^0[\hat{R}_n]&=\gamma \sum_{k=1}^n \PP^0[\{\zeta_k~\hbox{is a ladder epoch}\}]=\gamma \E^0[R_n].
\end{align}. 

By the Lemma \ref{lem2speedrw} and (\ref{eq1lemma3rw}), 
\begin{align}
\label{eq2lemma3rw}
\lim_{n\to \infty} \frac{\E^0[\hat{R}_n]}{n}=\frac{\gamma}{\rho}.
\end{align}
Next, notice that, by definition, $\E^0[R_n]\leq n$ and $\upsilon_n\geq n$, so, using $(\ref{eq2lemma3rw})$,
\begin{align}
\label{eq3lemma3rw}
\lim_{n\to \infty} \frac{\sum_{i=1}^{n} (\upsilon_i-\upsilon_{i-1})}{n}&=\lim_{n\to \infty} \frac{\upsilon_n}{n} \leq \lim_{n\to \infty} \frac{n}{\E^0[\hat{R_n}]}\leq \frac{\rho}{\gamma}. 
\end{align}
Now $\{\upsilon_n-\upsilon_{n-1}\}_{n>0}$ is an i.i.d. sequence. Hence, by the strong law of large numbers, we conclude from \eqref{eq3lemma3rw} that $\E^0[\upsilon_1-\upsilon_0]<\infty$. 
\end{proof}

\begin{proof}[Proof of Proposition \ref{propspeedrw}]

First, suppose $D_0$ has a $\PP^0-$tail heavier than $L_0$. As the asymptotic efficiency when tasks may be repeated  is less or equal to the asymptotic efficiency of when this is not the case, $e_p$, the former exists and it is zero. 

Now suppose $L_0$ has a strict $\PP^0-$tail heavier than $D_0$. Let 
\begin{align*}
\mathbb{T}_{n}^R=\sum_{j=\upsilon_{n-1}+1}^{\upsilon_n} T_{j}^R,
\end{align*} 
for $n>0$. Notice that 
$
\lim_{n\to \infty} \frac{1}{n}\sum_{i=1}^n T^R_n=\lim_{n\to \infty} \frac{1}{n}\sum_{i=1}^n \mathbb{T}_{n}^R.$

Now, by Lemma \ref{lem3speedrw}, since $\{\upsilon_n-\upsilon_{n-1}\}_{n\geq 0}$ is an $\hbox{i.i.d.}$ sequence, $\Phi$ is a renewal process, the sequence $\{\mathbb{T}_{n}^R\}_{n\ge 1}$ is i.i.d.. 

The next step is to show that $\E^0[\mathbb{T}_{0}^R]$ is finite as long as $L_0$ has a strict $\PP^0-$tail heavier than $L_0$, and can be bounded from above. First, notice that, since the discrete left-shift $\theta$ preserves $\PP^0$ and the random walk is independent of $\Phi$,  
\begin{align*}
\E^0[\mathbb{T}_{0}^R]&=\E^0[\mathbb{T}_{0}^R\circ \theta^{-\upsilon_0}]=\E^0\left[\sum_{i=1}^{\upsilon_1-\upsilon_0} T_{i}^R\right].
\end{align*}

Now we use the following version of Wald's equality. 
\begin{lem}
\label{Waldsequation}
If $\{Z_j\}_{j\in \bbN}$ is a sequence of positive random variables and $\eta$ is a positive integer-valued random variables satisfying 
\begin{enumerate}
\item $\E^0[Z]:=\E^0[Z_j]<\infty$ for all $j\in \bbN$;
\item $\E^0[Z_j\textbf{1}\{\eta \geq j\}]=\E^0[Z]\PP^0[\eta\geq j]$ for all $j\in \bbN$,
\end{enumerate}
then
\begin{align*}
\E^0\left[\sum_{i=1}^{\eta} Z_i\right]= \E^0[Z]\E^0[\eta].
\end{align*}
\end{lem} 

Since $\Phi$ is a renewal process and the $\{L_i\}_{i\geq 1}$ is an $\hbox{i.i.d.}$ sequence, $\E^0[T_{j}^R]=\E^0[T_{i}^R]$ for all $i,j\in \bbN$. Moreover, if $\E^0[T^R_0]<\infty$ Assumption 1 of Lemma \ref{Waldsequation} is satisfied. Assumption 2 of the same lemma is satisfied since the random walk $\{\zeta_k\}_{k\in \bbN}$ is independent of the point process and of the sequence $\{L_i\}_{i\geq 1}$. 

As $\{\upsilon_n-\upsilon_{n-1}\}_{n>0}$ is an $\hbox{i.i.d.}$ sequence (Lemma \ref{lem1speedrw}), and $\E^0[\upsilon_0-\upsilon_1]<\infty$, by the strong law of large numbers, 
\begin{align*}
\lim_{n\to \infty} \frac{1}{N}\sum_{i=0}^{N-1} \hat{\mathbb{T}}^R_n=\E^0\left[\sum_{i=\upsilon_0+1}^{\upsilon_1} \hat{T}^R_i\right]= \E^0[T^R_0]\E[\upsilon_1-\upsilon_0]<\infty,~~\PP-\hbox{a.s.}. 
\end{align*}

Therefore, we conclude that 
\begin{align*}
e_p=\frac{\E^0[D_0]}{\E^0[T^R_0]\E^0[\upsilon_1-\upsilon_0]}, \PP-\hbox{a.s.}.
\end{align*} 

By (\ref{eq3lemma3rw}), $\E^0[\upsilon_1-\upsilon_0]\leq \frac{\rho}{\gamma}$, completing the proof. 
\end{proof}

\appendix

\section{Proof of Theorem \ref{assymptotics1}}
\label{proofsappendix} 
Let $T^R(z)$ be the \textit{actual} \re time given $D=z$. Then 
\begin{align*}
T^R(z)=z\textbf{1}\{L_0>z\}+(\hat{T}_R(z)+L_0)\textbf{1}\{L_0\le z\},
\end{align*}	
where $\hat{T}^R(z)$ is an independent copy of $T^R(z)$. Let $m_R(z)$ be the expectation of the actual \re time given $D=z$. Then
\begin{align*}
m_R(z)&=z\PP[L_0>z]+m_R(z)\PP[L_0\le z|D=z]+\E[L_0\textbf{1}\{L_0\le z\}|D=z],
\end{align*}
so, as $D$ is independent of $L_0$
\begin{align*}
m_R(z)&=z+\frac{\E[L_0\textbf{1}\{L_0\le z\}|D=z]}{\PP[L_0>z]}.
\end{align*}
Therefore,
\begin{align*}
\E[T^R]&=\E[m_R(z)]=\E[D]+\int_0^{\infty}\frac{\E[L_0\textbf{1}\{L_0\le z\}]f_D(dz)}{\PP[L_0>z]}.
\end{align*}
	
In the same vein, let let $T^C(z)$ be the \textit{actual} \ch time given $D=z$. Then $T^C(z)=L_0\textbf{1}\{L_0>z\}+(\hat{T}_R(z)+L_0)\textbf{1}\{L_0\le z\}$, where 
$\hat{T}_C(z)$ is an independent copy of $T^C(z)$. Let $m_C(z)$ be the expectation of the actual \ch time given $D=z$. Then
\begin{align}
m_C(z)&=\frac{\E[L_0]}{\PP[L_0>z]}\nonumber.
\end{align}
Therefore, 
\begin{align*}
\E[T^C]=\int_{0}^{\infty}\frac{\E[L_0]f_D(dz)}{\PP[L_0>z]}. 	
\end{align*}

By construction, $T^C\geq T^R$, $\PP-$a.s.. Suppose $D$ has a $\PP-$tail heavier than $L_0$, i.e., there exist $z_0>0$ so that $\PP[L_0>z]\leq \PP[D>z]$ for all $z\geq z_0$. We show that $\E[T^R]=\infty$, which implies that $\E[T^C]=\infty$ as well. 

For any $z_0>0$, let
\begin{align*}
Y(z_0)&:=\int_{z_0}^{\infty}\frac{\E[L_0\textbf{1}\{L_0\leq z\}]f_D(dz)}{\PP[L_0>z]}.
\end{align*}
Pick $z_0$ such that $\PP[L_0>z]\leq \PP[D>z]$ for all $z\geq z_0$. Then
\begin{align*}
Y(z_0)&\geq \int_{z_0}^{\infty}\frac{\E[L_0\textbf{1}\{L_0\leq z\}]f_D(dz)}{\PP[D>z]}\\
&\geq \E[L_0\textbf{1}\{L_0\leq z_0\}]\int_{z_0}^{\infty}\frac{f_D(dz)}{\PP[D>z]} .
\end{align*}
Letting $w=\PP[D\le z]$ we have:
\begin{align*}
Y(z_0)&\geq \E[L_0\textbf{1}\{L_0\leq z_0\}]\int_{w_0}^{1}\frac{1}{1-w}dw=\infty, 
\end{align*}
where $w_0=\PP[D\le z_0]>0$. As $\E[T^R]\geq Y(z_0)$, the result follows.

Now suppose $L_0$ has a strict $\PP-$heavier tail than $D$. In this case, it suffices to show that $\E[T^C]<\infty$. Indeed, suppose there exists $z_0$ and $0<\epsilon<1$ such that $\PP[L_0>z]\geq \PP[D>z]^{\epsilon}$ for all $z_0>z$.   Then
\begin{align*}
W(z_0)&:=\int_{z_0}^{\infty}\frac{f_D(dz)}{\PP[L_0>z]}\leq  \int_{z_0}^{\infty}\frac{f_D(dz)}{\PP[D>z]^{\epsilon}}.	
\end{align*}
Letting $w=\PP[D\leq z]$ we have:
\begin{align*}
W(z_0)&\leq \int_{w_0}^{1}\frac{1}{(1-w)^{\epsilon}}dw<\infty,
\end{align*}
which implies $\E[T^C]< \infty$.

\section*{Acknowledgments}
This work was supported by a grant of the Simons Foundations (\#197982 to
the University of Texas at Austin). I would like to thank Abishek Sankararaman for his very valuable comments on this work.  

\bibliographystyle{amsplain}
\bibliography{candidacy3} 

\end{document}